\documentclass[reqno,11pt]{amsart}
\usepackage{amsmath,amsfonts,amssymb,amsxtra,latexsym,amscd,enumerate,amsthm,verbatim}

\usepackage{graphicx}

\usepackage[margin=1.40in]{geometry}
\numberwithin{equation}{section}

\newcommand{\R}{\mathbb{R}}

\newcommand{\T}{\mathbb{T}}
\newcommand{\C}{\mathbb{C}}
\newcommand{\Z}{\mathbb{Z}}

\numberwithin{equation}{section} 

\newtheorem{theorem}{Theorem}[section]
\newtheorem{lemma}[theorem]{Lemma}
\newtheorem{proposition}[theorem]{Proposition}

\newtheorem{remark}[theorem]{Remark}

\newtheorem{assumption}[theorem]{Assumption}

\begin{document}

\title{Uniform linear inviscid damping and enhanced dissipation near monotonic shear flows in high Reynolds number regime (I): the whole space case}

\author{Hao Jia}
\address{University of Minnesota}
\email{jia@umn.edu}


\thanks{ Supported in part by NSF grant DMS-1945179.}

\begin{abstract}
{\small}

We study the dynamics of the two dimensional Navier Stokes equations linearized around a strictly monotonic shear flow on $\T\times\R$. The main task is to understand the associated Rayleigh and Orr-Sommerfeld equations, under the natural assumption that the linearized operator around the monotonic shear flow in the inviscid case has no discrete eigenvalues. We obtain precise control of solutions to the Orr-Sommerfeld equations in the high Reynolds number limit, using the perspective that the nonlocal term can be viewed as a compact perturbation with respect to 
the main part that includes the small diffusion term. As a corollary,  we give a detailed description of the linearized flow in Gevrey spaces (linear inviscid damping) that are uniform with respect to the viscosity, and enhanced dissipation type decay estimates. The key difficulty is to accurately capture the behavior of the solution to Orr-Sommerfeld equations in the critical layer. In this paper we consider the case of shear flows on $\T\times\R$. The case of bounded channels poses significant additional difficulties, due to the presence of boundary layers, and will be addressed elsewhere.
\end{abstract}

\maketitle

\setcounter{tocdepth}{1}
\pagestyle{plain}

\tableofcontents

\section{Introduction and main results}\label{sec:imr}
The study of stability problems in mathematical analysis of fluid dynamics has a long and distinguished history, dating back to the work of Kelvin \cite{Kelvin}, Orr \cite{Orr} and Rayleigh \cite{Ray} among many others, and continuing to the present day. Hydrodynamical stability problems can be considered in both two and three dimensions. In this paper we work with two dimensional flows. An important early theme was to understand when physically relevant steady flows may become unstable by studying discrete eigenvalues of the corresponding linearized operator, the presence of which often leads to exponential instability.  When the linearized operator is spectrally stable, it is natural to ask if we can prove nonlinear stability, which is the case for many parabolic equations.\footnote{We should mention that there are many important nonlinear stability results for fluid equations in Lyapunov or orbital sense based on variational argument, see Arnold \cite{Arnold} and recent works \cite{Choi, GallaySverak} for more references. Our focus is dynamic stability, or asymptotic stability, which requires a more precise understanding on the evolution of solutions.} The answer to this question turns out to be complicated due to many reasons. The main issue is that very often the most interesting physical settings involve high Reynolds number (or equivalently the viscosity has to be taken as very small), and the Navier Stokes equations degenerate to Euler equations where continuous spectrum plays a dominant role.  It is therefore important to understand mathematically the property of the linearized operator in the high Reynolds number regime, including the limiting inviscid problem. 

For the Euler equations, there are significant recent progresses on the asymptotic stability of shear flows and vortices, assuming spectral stability, see for example \cite{Grenier,dongyi,Zillinger1,Zillinger2, JiaL,JiaG,Bed2,JiaVortex,Xiaoliu,Stepin} for linear results. The main mechanism of stabilization is the so called ``inviscid damping", which refers to the transfer of energy of vorticity to higher and higher frequencies leading to decay of the stream and velocity functions, as $t\to\infty$.  Extending the linearized stability analysis for inviscid fluid equations to the full nonlinear setting is a challenging problem, and the only available results are on spectrally stable monotonic shear flows, see \cite{BM,NaZh,IOJI,IJacta}, and on point vortices \cite{IOJI2}. We refer also to the recent review article \cite{IJICM} for a more in-depth discussion of recent developments of both linear and nonlinear inviscid damping. 

 When the viscosity $\nu>0$ is small but nonzero, there is another important physical phenomenon, called enhanced dissipation, which helps to stabilize the flow. Roughly speaking, the background flow mixes the vorticity field which makes the viscous effect more powerful in averaging the vorticity function, leading to faster decay of the vorticity than when viscous effect acts alone. The enhanced dissipation can be used to establish improved nonlinear asymptotic stability results for perturbations of the order $O(\nu^\gamma)$ for a suitable $\gamma>0$. The determination of the smallest $\gamma$ for which nonlinear stability still holds is an important problem, and is an active research area, aiming to address the ``transition threshold conjecture".  We refer the reader to \cite{Helffer,CLWZ,CWZ,CEW,Gallay,Gallay2,Gallagher,LWZ20,MZ22,MZ20,Zhiwu3,Wang,Dongyi3,Dongyi4} and references therein for important recent works on enhanced dissipation and transition threshold problems, and section 6.11 of the recent book \cite{BeVi} for an excellent survey.

In view of the results on linear inviscid damping for the Euler equation, it is natural to expect that even for Navier Stokes equations with small viscosity we should have the same inviscid damping results with explicit rates of decay in time of the stream functions and velocity fields. Surprisingly, results in this direction are very few, as pointed out in \cite{BeVi}. The only results addressing precise uniform inviscid damping, as far as we know, are \cite{CLWZ,MZ22,BH20, Bed6} for the Couette flow with full nonlinear analysis or precise linear results, and \cite{Grenier} for the spectrally restricted stream function for ``mixing layer" type shear flows (but a description of the full solution without spectral restrictions is not available).

The main reason for this gap is that the Navier Stokes equations in high Reynolds number regime is a singular perturbation of the Euler equations, and the corresponding analysis for establishing precise inviscid damping is significantly more complicated. As a consequence, the problem remains largely open for more general shear flows than the Couette flow.   

To bridge this gap, we take a first step and study the linear asymptotic stability of monotonic shear flows $(b(y),0)$ on $\T\times\R$, and obtain precise inviscid damping estimates for the linearized flow that are uniform with respect to viscosity, together with enhanced dissipation estimates. Strictly speaking, $(b(y),0)$ is not a steady state for the Navier Stokes equations, and becomes steady state only with a small external force $f=(-\nu b''(y),0)$. In our setting, we consider the viscosity $\nu$ to be very small. The effect of diffusion on the background shear flow is negligible, at least up to the diffusion time scale $T\ll \frac{1}{\nu}$. After the diffusion time, due to the enhanced dissipation effect, the flow is essentially dominated by a heat evolution, at least heuristically. Therefore our analysis below still captures the main dynamics of the Navier Stokes equations near shear flows, even when we do not add the external forcing.  We hope the methods introduced here can be applied to establish uniform inviscid damping in high Reynolds number regime for more complicated and physically relevant flows, such as the Poiseuille flow in a periodic channel and monotone vortices in $\R^2$. 

Our main assumptions on the background flow $(b(y),0)$ are the following.
\begin{assumption}[Main assumption on the shear flow $(b(y),0)$]\label{MaAs}
We assume that the shear flow $(b(y),0)$ satisfies the following conditions:
\begin{itemize}
\item For some $\sigma_0\in(0,1)$, we have
\begin{equation}\label{asp1}
b'(y)\in[\sigma_0,1/\sigma_0]\,\,{\rm for\,\,}y\in\R,\quad {\rm supp}\,b''\subseteq[-1/\sigma_0,1/\sigma_0], \quad\sup_{\xi\in\R}\Big[e^{\sigma_0\langle\xi\rangle^{1/2}}|\widehat{\,\,b''}(\xi)|\Big]\leq 1/\sigma_0;
\end{equation}

\item The linearized operator $L_k: L^2(\R)\to L_{\rm loc}^2(\R)$ with $k\in\Z\backslash\{0\}$ defined for $g\in L^2(\R)$ as
\begin{equation}\label{asp2}
L_kg(y):=b(y)g(y)-b''(y)\varphi, \quad{\rm where}\,\,(-k^2+\partial_y^2)\varphi(y)=g(y),\,\,y\in\R,
\end{equation}
has purely continuous spectrum $\R$. 

\end{itemize}
\end{assumption}

\subsection{Main equations}
Assume that $\nu\in(0,1)$. The main linearized equation around the shear flow $(b(y),0)$ we shall consider is 
\begin{equation}\label{int1}
\left\{\begin{array}{rl}
\partial_t\omega-\nu\Delta\omega+b(y)\partial_x\omega-b''(y)\partial_x\psi=0,&\\
\Delta\psi=\omega,&
\end{array}\right.
\end{equation}
for $(x,y,t)\in\T\times\R\times[0,\infty)$. Taking the Fourier transform in $x$, we obtain for each $k\in\Z$ that
\begin{equation}\label{int2}
\left\{\begin{array}{rl}
\partial_t\omega_k+\nu(k^2-\partial_y^2)\omega_k+ikb(y)\omega_k-ikb''(y)\psi_k=0,&\\
(-k^2+\partial_y^2)\psi_k=\omega_k,&
\end{array}\right.
\end{equation}
for $(y,t)\in\R\times[0,\infty)$.
 For $k\in\Z\backslash\{0\}$, define the operator $L_{k,\nu}: H^2(\R)\to L_{\rm loc}^2(\R)$ as follows. For $g\in H^2(\R)$ and $y\in\R$,
\begin{equation}\label{int5}
L_{k,\nu}g(y):=(\nu/k)\partial_y^2g-ib(y)g+ib''(y)\varphi, \quad{\rm with}\,\,(-k^2+\partial_y^2)\varphi=g.
\end{equation}
We can rewrite equation \eqref{int2} more abstractly as
\begin{equation}\label{int3}
\partial_t\omega_k^\ast=kL_{k,\nu}\omega_k^\ast, \quad{\rm for}\,\,t\ge0,
\end{equation}
where we have set
\begin{equation}\label{int4}
\omega_k^\ast(t,y)=e^{\nu k^2t}\omega_k(t,y).
\end{equation}
By spectral theory, we have the following representation formula for $y\in\R$ and $t>0$,
\begin{equation}\label{int6}
\begin{split}
\omega_k^\ast(t,y)&=\frac{1}{2\pi i}\int_{i\R} e^{\mu tk}\big[(\mu-L_{k,\nu})^{-1}\omega_{0k}\big](y)\,d\mu\\
&=\frac{1}{2\pi }\int_{\R} e^{i\lambda tk}\big[(i\lambda-L_{k,\nu})^{-1}\omega_{0k}\big](y)\,d\lambda\\
&=-\frac{1}{2\pi }\int_{\R} e^{-ib(y_0) tk}\big[(ib(y_0)+L_{k,\nu})^{-1}\omega_{0k}\big](y)b'(y_0)\,dy_0.
\end{split}
\end{equation}
The formula \eqref{int6} holds for $\omega_{0k}\in C_0^\infty(\R)$, and can be derived using methods from the spectral theory of sectorial operators, with slight modifications. We shall give a more detailed discussion in Proposition \ref{rep1}. We remark that since we will obtain quantitative bounds on the vorticity and stream functions that are independent of the support of $\omega_{0k}$, our main conclusions do not require the compact support assumption needed to establish the formula \eqref{int6}, by a standard approximation argument.
Define for $y,y_0\in\R$,
\begin{equation}\label{int7}
\omega_{k,\nu}(y,y_0)=\big[(ib(y_0)+L_{k,\nu})^{-1}\omega_{0k}\big](y).
\end{equation}
It follows from \eqref{int7} that 
\begin{equation}\label{int8}
ib(y_0)\omega_{k,\nu}(y,y_0)+L_{k,\nu}\omega_{k,\nu}(y,y_0)=\omega_{0k}(y), \quad{\rm for}\,\,y,y_0\in\R.
\end{equation}
Therefore $\omega_{k,\nu}(y,y_0)$ satisfies the equation for $k\in\Z\backslash\{0\}$, $y,y_0\in\R$,
\begin{equation}\label{int9}
\left\{\begin{array}{rl}
\frac{\nu}{k}\partial_y^2\omega_{k,\nu}(y,y_0)+i(b(y_0)-b(y))\omega_{k,\nu}(y,y_0)+ib''(y)\psi_{k,\nu}(y,y_0)=\omega_{0k}(y),&\\
(-k^2+\partial_y^2)\psi_{k,\nu}(y,y_0)=\omega_{k,\nu}(y,y_0).&
\end{array}\right.
\end{equation}
To obtain sharp Gevrey estimates for the profiles of the vorticity and stream function, it is important to work with the variables $v, w\in\R$, defined as 
\begin{equation}\label{int10}
v=b(y), \quad w=b(y_0), \quad B^\ast(v)=b'(y),\quad{\rm for}\,\,y,y_0\in\R.
\end{equation}
We define with the change of variables \eqref{int10} for $y,y_0\in\R$,
\begin{equation}\label{int11}
\Pi_{k,\nu}(v,w):=\psi_{k,\nu}(y,y_0), \quad \Omega_{k,\nu}(v,w):=\omega_{k,\nu}(y,y_0),\quad F_{0k}(v):=\omega_{0k}(y),\quad f_k(t,v):=\omega_k(t,y).
\end{equation}
It follows from \eqref{int9}-\eqref{int11} that $\Pi_{k,\nu}$ and $\Omega_{k,\nu}$ satisfy for $v,w\in\R$, 
\begin{equation}\label{int12}
\begin{split}
\frac{\nu}{k}(B^\ast(v))^2\partial_v^2\Omega_{k,\nu}(v,w)+\frac{\nu}{k} B^\ast(v)\partial_vB^\ast(v)\partial_v\Omega_{k,\nu}(v,w)-i(v-w)\Omega_{k,\nu}(v,w)&\\
\qquad+iB^\ast(v)\partial_vB^\ast(v)\Pi_{k,\nu}(v,w)=F_{0k}(v),&\\
-\Big[k^2-(B^\ast(v))^2\partial_v^2-B^\ast(v)\partial_vB^\ast(v)\partial_v\Big]\Pi_{k,\nu}(v,w)=\Omega_{k,\nu}(v,w).&
\end{split}
\end{equation}
We define the ``profile" for the spectral density function $\Pi_{k,\nu}(v,w)$ with $v,w\in\R$ (see also \cite{JiaVortex} for a related definition) as
\begin{equation}\label{int13}
\Theta_{k,\nu}(v,w):=\Pi_{k,\nu}(v+w,w).
\end{equation}
Using the formula \eqref{int6}, the identity \eqref{int4} and the change of variable \eqref{int10}-\eqref{int11}, we obtain the representation formula
\begin{equation}\label{int13.1}
f_k(t,v)=-\frac{1}{2\pi}e^{-\nu k^2t}\int_{\R}e^{-ikwt}\Omega_{k,\nu}(v,w)\,dw,\quad{\rm for}\,\,v\in\R, t>0.
\end{equation}

We summarize our calculations in the following proposition.
\begin{proposition}\label{intP1}
Assume that $\nu\in(0,1/10)$ and $k\in\Z\backslash\{0\}$. Suppose that $\omega_k(t,y)$ satisfying the regularity condition $\omega_k(t,y)\in C^\infty((0,\infty)\times\R)$ and $\omega_k(t,\cdot)\in C([0,\infty), L^2(\R))$ is the solution to the system of equations 
\begin{equation}\label{intP2}
\left\{\begin{array}{rl}
\partial_t\omega_k(t,y)+\nu(k^2-\partial_y^2)\omega_k(t,y)+ikb(y)\omega_k(t,y)-ikb''(y)\psi_k(t,y)&=0,\\
(-k^2+\partial_y^2)\psi_k(t,y)&=\omega_k(t,y),
\end{array}\right.
\end{equation}
for $(y,t)\in\R\times[0,\infty)$, with initial data $\omega_k(0,y)=\omega_{0k}(y)\in C_0^\infty(\R)$. Define the operator $L_{k,\nu}: H^2(\R)\to L_{\rm loc}^2(\R)$ as the following 
\begin{equation}\label{intP3}
L_{k,\nu}g(y)=(\nu/k)\partial_y^2g-ib(y)g-i\frac{b''(y)}{|k|}\int_\R e^{-|k||y-z|}g(z)\,dz, \quad{\rm for\,\,any\,\,}g\in H^2(\R). 
\end{equation}
 Set for $y,y_0\in\R$
\begin{equation}\label{intP4}
\omega_{k,\nu}(y,y_0)=\big[(ib(y_0)+L_{k,\nu})^{-1}\omega_{0k}\big](y).
\end{equation}
Then $\omega_{k,\nu}(y,y_0)$ satisfies the equation for $k\in\Z\backslash\{0\}$, $y,y_0\in\R$,
\begin{equation}\label{intP5}
\left\{\begin{array}{rl}
\frac{\nu}{k}\partial_y^2\omega_{k,\nu}(y,y_0)+i(b(y_0)-b(y))\omega_{k,\nu}(y,y_0)+ib''(y)\psi_{k,\nu}(y,y_0)=\omega_{0k}(y),&\\
(-k^2+\partial_y^2)\psi_{k,\nu}(y,y_0)=\omega_{k,\nu}(y,y_0).&
\end{array}\right.
\end{equation}
We have the representation formula for $y\in\R, t>0$,
\begin{equation}\label{intP6}
\omega_k(t,y)=-\frac{1}{2\pi }e^{-\nu k^2t}\int_{\R} e^{-ikb(y_0)t}\omega_{k,\nu}(y,y_0)b'(y_0)\,dy_0.
\end{equation}
Define the change of variables for $y,y_0\in\R$ and $t\ge0$,
\begin{equation}\label{intP7}
\begin{split}
&v=b(y), \quad w=b(y_0), \quad B^\ast(v)=b'(y),\quad F_{0k}(v)=\omega_{0k}(y),\quad \Pi_{k,\nu}(v,w)=\psi_k(y,y_0), \\
& \Omega_{k,\nu}(v,w)=\omega_{k,\nu}(y,y_0),\,\, \Theta_{k,\nu}(v,w)=\Pi_{k,\nu}(v-w,w),\,\, f_k(t,v)=\omega_k(t,y),\,\, \phi_k(t,v)=\psi_k(t,y).
\end{split}
\end{equation}
Then $\Pi_{k,\nu}$ and $\Omega_{k,\nu}$ satisfies for $v,w\in\R$, 
\begin{equation}\label{intP8}
\begin{split}
\frac{\nu}{k}(B^\ast(v))^2\partial_v^2\Omega_{k,\nu}(v,w)+\frac{\nu}{k} B^\ast(v)\partial_vB^\ast(v)\partial_v\Omega_{k,\nu}(v,w)-i(v-w)\Omega_{k,\nu}(v,w)&\\
\qquad+iB^\ast(v)\partial_vB^\ast(v)\Pi_{k,\nu}(v,w)=F_{0k}(v),&\\
-\Big[k^2-(B^\ast(v))^2\partial_v^2-B^\ast(v)\partial_vB^\ast(v)\partial_v\Big]\Pi_{k,\nu}(v,w)=\Omega_{k,\nu}(v,w).&
\end{split}
\end{equation}
Moreover, we have the representation formula for $v\in\R,t>0$,
\begin{equation}\label{intP13.1}
f_k(t,v)=-\frac{1}{2\pi}e^{-\nu k^2t}\int_{\R}e^{-ikwt}\Omega_{k,\nu}(v,w)\,dw,
\end{equation}
and
\begin{equation}\label{intP13.2}
\phi_k(t,v)=-\frac{1}{2\pi}e^{-\nu k^2t}\int_{\R}e^{-ikwt}\Theta_{k,\nu}(v-w,w)\,dw.
\end{equation}

\end{proposition}

Strictly speaking the vorticity function $\omega_k(t,y)$ and stream function $\psi_k(t,y)$, together with their variants in the variable $v$, also depend on $\nu$. We omit this dependence from our notations for the sake of simplicity, since there is no danger of confusion.

\subsection{Main results}
Our main results are sharp regularity estimates in Gevrey spaces, and enhanced dissipation, for the ``profile" of the vorticity function $f_k(t,v)$. Below we allow the implied constants to depend on $\sigma_0$ from \eqref{asp1} and the structural constant $\kappa>0$ from \eqref{LAP23} connected to the limiting absorption principle.

\begin{theorem}\label{intM1}
Assume the notations and conditions in Proposition \ref{intP1}. Then there exists a constant $\nu_0\in(0,1)$ such that the following statement holds. Assume that $\nu\in(0,\nu_0)$ and $0<\delta\ll \sigma_0$. Define for $v\in\R, t\ge0$,
\begin{equation}\label{intM2}
f_k(t,v):=F_k(t,v)e^{-ikvt},\qquad \phi_k(t,v):=\Phi_k(t,v)e^{-ikvt}.
\end{equation}
Then for suitable $c_0\in(0,1)$, we have the bounds for $t\ge0$
\begin{equation}\label{intM3}
\left\|F_k(t,\cdot)\right\|_{L^2(\R)}\lesssim e^{-c_0\nu^{1/3}|k|^{2/3}t}e^{-\nu k^2t}\left\|F_{0k}\right\|_{L^2(\R)}.
\end{equation}
In addition, we have the uniform bounds for $t\ge0$,
\begin{equation}\label{intM4}
\left\|e^{\delta\langle k,\xi\rangle^{1/2}}\widehat{F_k}(t,\xi)\right\|_{L^2(\R)}\lesssim e^{-\nu k^2t}\left\|e^{\delta\langle k,\xi\rangle^{1/2}}\widehat{\,\,F_{0k}}(\xi)\right\|_{L^2(\R)}.
\end{equation}
Moreover, the profile for the stream function satisfies the bounds for $t\ge0$,
\begin{equation}\label{intM5}
\begin{split}
 \left\|\langle k,\xi-kt\rangle^2 e^{\delta\langle k,\xi\rangle^{1/2}}\widehat{\,\,\Phi_k}(t,\xi)\right\|_{L^2(\R)}\lesssim e^{-\nu k^2t}\left\|e^{\delta\langle k,\xi\rangle^{1/2}}\widehat{\,\,F_{0k}}(\xi)\right\|_{L^2(\R)}.
 \end{split}
\end{equation}
\end{theorem}

The proof of the theorem is based on the representation formula \eqref{intP13.1}-\eqref{intP13.2} and will be given in section \ref{sec:pmt}. The main task is to obtain the following precise control on the ``profile" $\Theta_{k,\nu}(v,w)$ of the spectral density function in Gevrey spaces. 
\begin{proposition}\label{int14}
There exists a constant $\nu_0\in(0,1/8)$ sufficiently small, such that the following statement holds. For $k\in\Z\backslash\{0\}$, $\nu\in(0,\nu_0)$ and $0<\delta\ll \sigma_0$, the profile for the spectral density function, $\Theta_{k,\nu}(v,w)$ satisfies the bounds 
\begin{equation}\label{int15}
\Big\|(|k|+|\xi|)e^{\delta \langle k,\eta\rangle^{1/2}}\widetilde{\,\,\Theta_{k,\nu}}(\xi,\eta)\Big\|_{L^2(\R^2)}\lesssim \Big\|e^{\delta\langle k,\eta\rangle^{1/2}}\widehat{\,\,F_{0k}}(\eta)\Big\|_{L^2(\R)}.
\end{equation}
\end{proposition}

In the above, $\widehat{\,\,h\,\,}(\xi)$ is the Fourier transform of $h\in L^2(\R)$ and $\widetilde{\,\,\Theta_{k,\nu}}(\xi,\eta)$ is the Fourier transform of $\Theta_{k,\nu}(y,y_0)$ in $y,y_0\in\R$.

We note that \eqref{intM4} provides very precise bounds on the ``profile" of the vorticity function that are uniform with respect to the viscosity, and is our main conclusion, while \eqref{intM3} gives an enhanced dissipation estimate in comparison with the rate $e^{-C\nu t}$ that holds for general solutions to the heat equation. The bound \eqref{intM5} implies that the stream function decays quadratically in $t$ if the initial data is sufficiently smooth. The estimates \eqref{intM3}-\eqref{intM4} are strong enough for proving full nonlinear inviscid damping, using also the ideas in \cite{IJacta}, although considerable additional effort to control the nonlinearity and methods in treating the slowly time dependent background shear flow would be needed. 

\subsection{Main ideas and further problems}
\subsubsection{Main ideas of proof} We briefly describe our main ideas to establish Theorem \ref{intM1}.  With the representation formula \eqref{intP13.1}-\eqref{intP13.2}, the proof of the bounds \eqref{intM4}-\eqref{intM5} is reduced to the study of the Orr-Sommerfeld equations \eqref{intP5} (in $y,y_0$ variables), and also \eqref{intP8} (in $v,w$ variables).\footnote{The study of the Orr-Sommerfeld equation has a long history. We refer to \cite{Gre1,Gre2} for an introduction to other important aspects of the equation.} The main task is to establish the bound \eqref{int15} on the profile of the spectral density function $\Theta_{k,\nu}(v,w)$. 

Our argument, on a conceptual level, can be divided into two steps. The first important step is to obtain a very detailed understanding of the fundamental solution of the main Airy part $\frac{\nu}{k}\partial_y^2+i(b(y_0)-b(y))$ of the Orr-Sommerfeld equation, in Gevrey spaces. The characterization of the fundamental solution we need consists of both pointwise bounds, see \eqref{GKA2}-\eqref{GKA3}, and estimates in Gevrey regularity spaces after suitable renormalizations, see Proposition \ref{GKA11}. 

The second main step is to show that the full Orr-Sommerfeld equation can be viewed as a relatively compact perturbation of the main Airy part, and can be treated using a limiting absorption principe, see Proposition \ref{LAP22}, under the spectral assumption that the linearized operator in the inviscid case has no discrete eigenvalues, see Assumption \ref{MaAs}. 

Therefore, the general line of our argument is analogous to the study of the Rayleigh equation, except that we need to take the small-in-size but highest-in-order diffusion term as part of the main term. The essential complication is that we can no longer solve the main Airy equation explicitly, and instead need to rely on various pointwise bounds and regularity estimates on the fundamental solution. 

On a more technical level, to establish the required pointwise bounds on the fundamental solution $k^\ast_{\epsilon}(y,z;y_0)$ with $y, z, y_0\in\R$ of the main Airy part, defined for $y, z, y_0\in\R$ and $\epsilon\in(0,1)$ as
\begin{equation}\label{Intro:GKA6}
\Big[\epsilon\partial_y^2+i(b(y_0)-b(y))\Big]k^\ast_\epsilon(y,z;y_0)=\delta(y-z),
\end{equation}
we use several energy estimates which achieve different levels of control, see section \ref{sec:pbk}. The most interesting energy estimate is perhaps the one which we call ``entanglement inequalities", see Lemma \ref{BKGL31} and the remark below it, as these inequalities provide bounds on the fundamental solution in one region by its behavior in another, possibly far away, region. Such inequalities are not new, and can be found for example in \cite{Wei3}, although our setting is slightly different. 

The pointwise bounds alone are not sufficient for our purposes, due to the strong singularity of the fundamental solution at $y=y_0$ or equivalently $v=w$. To obtain a more accurate description of the singularity, as in \cite{IJacta,JiaG}, we perform a renormalization by making the change of variable $(v,w)\to (v+w,w)$ after reformulating the fundamental solution in $v, w$ variables, see \eqref{GKA6}, which shifts all singularity to the $v$ variable and the resulting kernel becomes very smooth in $w$. 

The main remaining difficulty is to capture the singular behavior of the kernel in $v$, which is quite complicated since the kernel is both oscillatory and singular. It appears that there is no simple way to renormalize such complicated behavior. To overcome this difficulty, we derive various characterizations of the singular behavior in $v$, see for example \eqref{LAPM4.1} and the decomposition \eqref{GKA14.1}. 

Once we have sufficient understanding on the kernel, we can establish the limiting absorption principle, see Proposition \ref{LAP22}, and obtain a preliminary bound on $\Theta_{k,\nu}(v,w)$ in low regularity Sobolev spaces. We then apply the Fourier multiplier associated with the symbol $e^{\delta\langle k,\eta\rangle^{1/2}}$ and use a commutator argument to get higher regularity bounds, as in \cite{IJacta,JiaG}. 

To obtain the enhanced dissipation bound \eqref{intM3}, in section \ref{sec:dsr} we generalize an important result from \cite{Wei3} which gives sharp decay bounds on the semigroup using resolvent estimates (that was inspired by the earlier work of Helffer and Sj\"ostrand \cite{Helffer2}). Our main observation here is that the assumption in \cite{Wei3} on the generator of the semigroup being accretive may be replaced by an a priori bound on the semigroup, which can be obtained by a more detailed analysis of the spectral density function. This new formulation appears applicable for a wide range of singular perturbation problems where sharp semigroup bounds are useful. 

After the completion of this work, we learnt about the recent work of Helffer and Sj\"ostrand \cite{Helffer3} which contained a more general version of our estimate (see Theorem \ref{DSR0}) on the decay of semigroups using resolvent bounds and a priori semigroup bounds, with sharp constants. We decide to keep the statement and proof of our semigroup bounds for their simplicity and for the sake of completeness, and refer to \cite{Helffer3} for the more general version.

\subsubsection{Further problems} A natural next step is to study the uniform inviscid damping near monotonic shear flows inside a bounded periodic channel, such as $\T\times[0,1]$. In this setting, an essential new difficulty is the appearance of a ``boundary layer" near $y\in\{0,1\}$, due to the mismatch between non-slip and non-penetration boundary conditions for the Navier Stokes and Euler equations, respectively. 

Another interesting problem is to consider shear flows that are not necessarily monotonic, such as Kolmogorov flows. For such flows on the linear inviscid level, there is an additional physical phenomenon called ``vorticity depletion" which refers to the asymptotic vanishing of vorticity as $t\to\infty$ near the critical point where the derivative of the shear flow is zero, first predicted in \cite{Bouchet} and proved in \cite{Dongyi2}. A similar phenomenon was proved in \cite{Bed2} for the case of vortices. See also \cite{JiaVortex} for a refined description of the dynamics as a step towards proving nonlinear vortex symmetrization. It is an important and very intriguing problem to understand, as precisely as possible, the interaction of inviscid damping, enhanced dissipation, vorticity depletion and boundary layers, in the high Reynolds number regime. 

In the above problems, although our method here will not apply in a straightforward fashion due to various additional difficulties, the general idea of a direct comparison between Orr-Sommerfeld equation in the high Reynolds number regime and Rayleigh equation through the limiting absorption principle may be applicable. 

It is also natural to ask whether the linear asymptotic stability results in the high Reynolds number regime can be extended to the full nonlinear problem, as in \cite{Bed6} for the Couette flow on $\T\times\R$. This is a very subtle and difficult problem in general, already at the inviscid level. In the setting considered in this paper, we believe that with additional work one can extend the nonlinear inviscid damping result to Navier Stokes equations with high Reynolds number, using the techniques developed here and in \cite{IJacta}.

\subsection{Notations and conventions}
Throughout the paper we shall assume that $k\in\Z\backslash\{0\}$ and $\nu\in(0,1/10)$. We shall use $e^{\delta_0\langle \xi\rangle^{1/2}}$ with a $\delta_0\in(0,1)$ depending on $\sigma_0$ from Assumption \ref{MaAs} to measure regularity of functions involving only the background flow. For example, we assume that 
$$\left\|e^{\delta_0\langle \xi\rangle^{1/2}}\widehat{\,\,\partial_vB^\ast}(\xi)\right\|_{L^2(\R)}\lesssim1.$$

We also assume that $0<\delta\ll \delta_0$, which will appear in the Fourier multiplier $e^{\delta\langle k,\xi\rangle^{1/2}}$ used to measure regularity of solutions. 

To fix constants, we use the following definition for Fourier transforms. For any $g\in L^2(\R)$,
\begin{equation*}
\widehat{\,\,g\,\,}(\xi):=\frac{1}{\sqrt{2\pi}}\int_\R f(y)e^{-iy\xi}\,dy.
\end{equation*}
For functions $g$ of more than one variable, we use the variables $y, z,\rho, v, w$ to denote physical variables for which we do not take the Fourier transform, and use the variables $\alpha, \beta, \gamma$ to denote Fourier variables, unless otherwise specified. We also use $\widetilde{g}$ to denote the partial or full Fourier transforms for functions of more than one variable. Therefore for example, 
\begin{equation*}
\widetilde{\,\,g\,\,}(v,\rho;\xi)=\frac{1}{\sqrt{2\pi}}\int_{\R}g(v,\rho;w)e^{-iw\xi}\,dw.
\end{equation*}

We use the convention that $ X\lesssim Y$ means $X\leq C Y$, where the implied constant $C$ is allowed to depend on $\sigma_0$. Dependence on additional parameters such as $\mu$ will be indicated by the notation $\lesssim_\mu$.

For $k\in\Z\backslash\{0\}$, we define the norm for any $f\in H^1(\R)$
\begin{equation*}
\|f\|_{H^1_k(\R)}:=|k|\|f\|_{L^2(\R)}+\|\nabla f\|_{L^2(\R)}.
\end{equation*}

Finally, we use the notation that for $m\in\Z\cap[1,\infty)$ and $a\in\R^m$, $\langle a\rangle:=(1+|a|^2)^{1/2}$.

\section{Preliminaries}
In this section, we prove several technical estimates that are useful for our applications in subsequent sections. We start with a localizing estimate for the Fourier multiplier with the symbol $e^{\mu\langle k,\xi\rangle^{1/2}}, \xi\in\R$.
\begin{lemma}\label{pre1}
Assume that $k\in\Z\backslash\{0\}$ and $\mu\in(0,1)$.  There exists a constant $c_0\in(0,1)$ independent of $k,\mu$ such that the following statement holds. Let $K(y)$ be the kernel of the Fourier multiplier $e^{\mu\langle k,\xi\rangle^{1/2}}$, $\xi\in\R$, which is a function for $|y|>1/c_0$ and can be calculated for $|y|>1/c_0$ as
\begin{equation}\label{pre2}
K(y)=\frac{1}{\sqrt{2\pi}}\lim_{R\to\infty}\int_\R e^{\mu\langle k,\xi\rangle^{1/2}}\Psi(\xi/R) e^{iy\xi}\,d\xi,
\end{equation}
where the limit is taken in the sense of distributions\footnote{Strictly speaking the test functions need to be at least Gevrey-2 smooth due to the strong growth of the symbol.} for $|y|>1/c_0$, and $\Psi$ is a Gevrey regular cutoff function, satisfying
\begin{equation}\label{pre3}
\Psi\in C_0^\infty(-2,2), \quad \Psi\equiv 1 \,\,{\rm on}\,\,[-1,1], \quad\sup_{\xi\in\R}\Big[e^{\langle\xi\rangle^{3/4}}\big|\widehat{\,\,\Phi\,\,}(\xi)\big|\Big]\lesssim1.
\end{equation}
Then we have the decay estimates for $|y|>1/c_0$,
\begin{equation}\label{pre4}
|K(y)|\lesssim e^{-c_0|y|^{1/2}}.
\end{equation}

\end{lemma}

\begin{proof}
Denote for $a\in\R$ and $N\in\Z\cap[1,\infty)$,
\begin{equation}\label{pre5}
\wp(a,N):=\prod_{j=1}^N\langle a-j+1\rangle.
\end{equation}
We first note the following pointwise bound which holds for a suitable $C_0\in(1,\infty)$ and all $N, m\in\Z\cap[1,\infty)$, $\xi\in\R$,
\begin{equation}\label{pre6}
\big|\partial_\xi^N\langle k,\xi\rangle^{m/2}\big|\leq C_0^N\langle k,\xi\rangle^{m/2-N}\wp(m/2,N) (N!).
\end{equation}
The inequality \eqref{pre6} can be proved by a standard induction argument. As a corollary of \eqref{pre6}, we also have for a suitable $C_1\in(1,\infty)$ and all $m\in\Z\cap[1,\infty)$, $N\in\Z\cap[8,\infty)$ and $\xi\in\R$,
\begin{equation}\label{pre7}
\begin{split}
\big|\partial_\xi^N\big[\langle k,\xi\rangle^{m/2}\Psi(\xi/R)\big]\big|\leq C_1^N\langle k, \xi\rangle^{m/2}\wp(m/2,N)\Big[\langle k,\xi\rangle^{-N}(N!)+R^{-1}\langle \xi\rangle^{-N+1} (N!)^{4/3}\Big].
\end{split}
\end{equation}
By the integration by parts formula, we can bound for $N\in\Z\cap[8,\infty)$ and $y\in \R$ with $|y|>1$,
\begin{equation}\label{pre8}
\begin{split}
&\Big|\lim_{R\to\infty}\int_\R \langle k,\xi\rangle^{m/2}\Psi(\xi/R) e^{iy\xi}\,d\xi\Big|\\
&=\Big|\lim_{R\to\infty}\int_\R y^{-N}\partial_\xi^N\Big[\langle k,\xi\rangle^{m/2}\Psi(\xi/R)\Big] e^{iy\xi}\,d\xi\Big|\lesssim \int_\R|y|^{-N}C_1^N\langle k,\xi\rangle^{m/2-N}\wp(m/2,N)(N!)\,d\xi.
\end{split}
\end{equation}
For $a\in\R$, denote
\begin{equation}\label{pre9}
I(a):=\max\{j\in\Z:\,\,j\leq a\}.
\end{equation}
We can achieve a rough optimization in \eqref{pre8} by choosing $N$ (depending on $m$) as follows.

{\it Case I: $|y|^{1/2}\gg m$.} In this case we set
\begin{equation}\label{pre10}
N=I\bigg(\frac{m}{4}+\sqrt{\frac{|y|}{4C_1}+\frac{m^2}{16}}\,\,\bigg).
\end{equation}
We obtain from \eqref{pre8} and \eqref{pre10} that for a suitable $c_0\in(0,1)$ and $y\in\R$ with $|y|>1/c_0$,
\begin{equation}\label{pre11}
\Big|\lim_{R\to\infty}\int_\R \langle k,\xi\rangle^{m/2}\Psi(\xi/R) e^{iy\xi}\,d\xi\Big|\lesssim m^{10}\big[(m/2)! \big]^2 e^{-c_0|y|^{1/2}}.
\end{equation}

{\it Case II: $|y|^{1/2}\lesssim m$.} In this case we simply set 
\begin{equation}\label{pre12}
N=I(m/2)+3.
\end{equation}
We obtain from \eqref{pre8} and \eqref{pre12} that for a suitable $c_0\in(0,1)$ and $y\in\R$ with $|y|>1/c_0$,
\begin{equation}\label{pre13}
\Big|\lim_{R\to\infty}\int_\R \langle k,\xi\rangle^{m/2}\Psi(\xi/R) e^{iy\xi}\,d\xi\Big|\lesssim m^{10}\big[(m/2)! \big]^2\lesssim (m!)e^{-c_0m}
 e^{-c_0|y|^{1/2}}.
\end{equation}

The desired bound \eqref{pre4} then follows from \eqref{pre11} and \eqref{pre13}, and a power series expansion of $e^{\mu\langle k,\xi\rangle^{1/2}}$.

\end{proof}

The following property for the Green's function of an elliptic operator that we shall study is important for many applications.
\begin{lemma}\label{pre200}
Assume that $k\in\Z\backslash\{0\}$. Let $\mathcal{G}_k(v,v')$ be the fundamental solution to the elliptic operator $-k^2+(B^\ast(v))^2\partial_v^2+B^\ast(v)\partial_vB^\ast(v)\partial_v$, in equivalence, for $v, v'\in\R$
\begin{equation}\label{pre201}
\Big[-k^2+(B^\ast(v))^2\partial_v^2+B^\ast(v)\partial_vB^\ast(v)\partial_v\Big]\mathcal{G}_k(v,v')=\delta(v-v').
\end{equation}
Fix any $\Psi_\ell\in C_0^\infty(-10,10)$ for $\ell\in\{1,2,3\}$ with
\begin{equation}\label{BSDG4.003}
\sup_{\xi\in\R,\ell\in\{1,2,3\}}\big|e^{\langle\xi\rangle^{3/4}}\widehat{\,\,\Psi_\ell}(\xi)\big|\lesssim1. 
\end{equation}

(i) For any $j\in\Z$, we define for $v,v'\in\R$,
\begin{equation}\label{BSDG4.004}
\mathcal{G}_{k}^j(v,v'):=\mathcal{G}_k(v,v')\Psi_1(v-v'-j)\Psi_2(v').
\end{equation}
Then for a suitable $\delta_0:=\delta_0(\sigma_0)\in(0,1)$, we have the bound for $j\in\Z$ and $\alpha,\beta\in\R$,
\begin{equation}\label{BSDG4.005}
\big|\widetilde{\,\,\mathcal{G}_{k}^j}(\alpha,\beta)\big|\lesssim e^{-\delta_0|j|}\frac{e^{-\delta_0\langle\alpha+\beta\rangle^{1/2}}}{k^2+\alpha^2}.
\end{equation}

Moreover for $v,v'\in\R$,
\begin{equation}\label{BSDG4.006}
|k|\big|\mathcal{G}_{k}(v,v')\big|+\big|\partial_v\mathcal{G}_{k}(v,v')\big|\lesssim e^{-\delta_0|k||v-v'|}.
\end{equation}

(ii) Let $\mathcal{G}_k(v,v';w)$ be the Green's function for the elliptic operator $-k^2+(B^\ast(v+w))^2\partial_v^2+B^\ast(v+w)\partial_vB^\ast(v+w)\partial_v$. More precisely, for $v,v',w\in\R$, 
\begin{equation}\label{BSDG4.001}
\Big[-k^2+(B^\ast(v+w))^2\partial_v^2+B^\ast(v+w)\partial_vB^\ast(v+w)\partial_v\Big]\mathcal{G}_k(v,v';w)=\delta(v-v').
\end{equation}

For any $j\in\Z$ and $w_0\in\R$, we define for $v,v',w\in\R$,
\begin{equation}\label{BSDG4.004}
\mathcal{G}_{k,w_0}^j(v,v';w):=\mathcal{G}_k(v,v';w)\Psi_1(w-w_0)\Psi_2(v-v'-j)\Psi_3(v').
\end{equation}
Then for a suitable $\delta_0=\delta_0(\sigma_0)\in(0,1)$, we have the bound for $j\in\Z$ and $\alpha,\beta, \xi\in\R$,
\begin{equation}\label{BSDG4.005}
\big|\widetilde{\,\,\mathcal{G}_{k,w_0}^j}(\alpha,\beta;\xi)\big|\lesssim e^{-\delta_0|j|}\frac{e^{-\delta_0\langle\alpha+\beta\rangle^{1/2}}}{k^2+\alpha^2}e^{-\delta_0\langle\xi\rangle^{1/2}}.
\end{equation}

(iii) Moreover for $v,v',w_0\in\R$ and $\xi\in\R$,
\begin{equation}\label{BSDG4.006}
|k|\big|\widetilde{\,\,\mathcal{G}_{k,w_0}}(v,v';\xi)\big|+\big|\partial_v\widetilde{\,\,\mathcal{G}_{k,w_0}}(v,v';\xi)\big|\lesssim e^{-\delta_0|k||v-v'|}e^{-\delta_0\langle\xi\rangle^{1/2}}.
\end{equation}
 
\end{lemma}

\begin{proof}
By the change of variables \eqref{int10}, we have the explicit formula for $v,v',w\in\R$,
\begin{equation}\label{BSDG4.0002}
\mathcal{G}_k(v,v')=-\frac{1}{|k|}\frac{e^{-|k||b^{-1}(v)-b^{-1}(v')|}}{B^\ast(v')}
\end{equation}
and
\begin{equation}\label{BSDG4.002}
\mathcal{G}_k(v,v';w)=-\frac{1}{|k|}\frac{e^{-|k||b^{-1}(v+w)-b^{-1}(v'+w)|}}{B^\ast(v'+w)}.
\end{equation}
The desired conclusions follow from \eqref{BSDG4.0002}-\eqref{BSDG4.002}, using the properties of Gevrey spaces, see Appendices A in \cite{IOJI} and \cite{JiaG} for more details.
\end{proof}

The following integral inequality will be used frequently below. 

\begin{lemma}\label{pre300}
Assume that $k\in\Z\backslash\{0\}$ and $\sigma\in(0,1)$. We have for $\epsilon\in(0,\sigma^3/2)$ and $v,\rho\in\R$,
\begin{equation}\label{BSDG4.007}
\int_{\R}e^{-\delta_0|k||v-v'|} \epsilon^{-1/3}\langle\epsilon^{-1/3}v',\epsilon^{-1/3}\rho\rangle^{1/2} e^{-\sigma\langle\epsilon^{-1/3}v',\epsilon^{-1/3}\rho\rangle^{1/2}\epsilon^{-1/3}|v'-\rho|}\,dv'\lesssim_\sigma  e^{-\delta_0|v-\rho|}.
\end{equation}

\end{lemma}

\begin{proof}
The proof follows from straightforward calculations.
\end{proof}


\section{Pointwise bounds on the kernel of generalized Airy operators}\label{sec:pbk}

In this section, we prove the following bounds on the kernel for the generalized Airy operator $\partial_Y^2+iV(Y)$ where $V(Y)\in C^1(\R)$ is real valued and ``non-stationary" in the sense that $|V'(Y)|\in[\sigma, 1/\sigma]$ for all $Y\in \R$ and some $\sigma\in(0,1)$.
\begin{proposition}\label{BKG1}
Assume that $\sigma\in(0,1)$, $V(Y)\in C^1(\R)$ with $V(0)=0$ and $|V'(Y)|\in[\sigma,1/\sigma]$ for all $Y\in\R$, and $\alpha\in[0,\infty)$. Let $K(Y,Z)$ be the fundamental solution to the generalized Airy operator $\partial_Y^2-\alpha+iV(Y)$. More precisely, for $Y, Z\in\R$, in the sense of distributions, 
\begin{equation}\label{BKG2}
\partial_Y^2K(Y,Z)-\alpha K(Y,Z)+iV(Y)K(Y,Z)=\delta(Y-Z).
\end{equation}
Then there exists $c_0\in(0,\infty)$ depending only on $\sigma$ such that
\begin{equation}\label{BKG3}
|K(Y,Z)|\lesssim_\sigma\frac{1}{\langle \alpha, Z\rangle^{1/2}}e^{-c_0\langle \alpha, Y, Z\rangle^{1/2}|Y-Z|},
\end{equation}
and
\begin{equation}\label{BKG4}
|\partial_YK(Y,Z)|\lesssim_\sigma\frac{\langle \alpha,Y\rangle^{1/2}}{\langle \alpha,Z\rangle^{1/2}}e^{-c_0\langle \alpha, Y, Z\rangle^{1/2}|Y-Z|}.
\end{equation}
In addition, for all $Z\in\R$,
\begin{equation}\label{BKG5}
|K(Z,Z)|\approx_{\sigma}\langle \alpha, Z\rangle^{-1/2}.
\end{equation}
\end{proposition}

The rest of the section is devoted to the proof of Proposition \ref{BKG1}. Without loss of generality we can assume that $V'(Y)\in[\sigma,1/\sigma]$. 


\subsection{Global energy estimates}
We begin with the following energy estimates.
\begin{lemma}\label{BKG6}
We have the energy bounds for all $Z\in\R$,
\begin{equation}\label{BKG6}
\int_\R \langle Y\rangle^2|K(Y,Z)|^2\,dY\lesssim_\sigma \langle \alpha,Z\rangle^{1/2}, 
\end{equation}
and
\begin{equation}\label{BKG7}
\alpha\int_\R |K(Y,Z)|^2\,dY+\int_\R |\partial_YK(Y,Z)|^2\,dY\lesssim_\sigma \langle \alpha,Z\rangle^{-1/2}.
\end{equation}
Moreover, for $Z\in\R$,
\begin{equation}\label{BKG8}
|K(Z,Z)|\approx_{\sigma}\langle \alpha, Z\rangle^{-1/2}.
\end{equation}

\end{lemma}

\begin{proof}
Multiplying $\overline{K(Y,Z)}$ to \eqref{BKG2} and integrating over $\R$, we obtain that 
\begin{equation}\label{BKG9}
-\int_\R |\partial_YK(Y,Z)|^2 dY-\alpha\int_\R |K(Y,Z)|^2\,dY+i\int_\R V(Y)|K(Y,Z)|^2 dY=\overline{K(Z,Z)}.
\end{equation}
Multiplying $V(Y)\overline{K(Y,Z)}$ to \eqref{BKG2} and integrating over $\R$, we obtain that 
\begin{equation}\label{BKG10}
\begin{split}
&-\alpha\int_\R V(Y)|K(Y,Z)|^2\,dY-\int_\R V(Y)|\partial_YK(Y,Z)|^2 dY-\int_\R V'(Y) \partial_YK(Y,Z) \, \overline{K(Y,Z)}\,dY\\
&+i\int_\R V^2(Y)|K(Y,Z)|^2 dY=V(Z)\overline{K(Z,Z)}.
\end{split}
\end{equation}
The identities \eqref{BKG9}-\eqref{BKG10} imply that 
\begin{equation}\label{BKG11}
\begin{split}
&\int_\R |\partial_YK(Y,Z)|^2 dY+\alpha\int_\R |K(Y,Z)|^2\,dY\leq|K(Z,Z)|,\\
&\int_\R V^2(Y)|K(Y,Z)|^2 dY\lesssim_\sigma |V(Z)K(Z,Z)|+\int_\R |\partial_YK(Y,Z)||K(Y,Z)|\,dY.
\end{split}
\end{equation}
As a consequence of \eqref{BKG11} and Poincar\'e inequality, we obtain that 
\begin{equation}\label{BKG12}
\begin{split}
&\int_\R \langle Y\rangle^2|K(Y,Z)|^2\,dY\lesssim_\sigma \langle Z\rangle|K(Z,Z)|+\int_\R |\partial_YK(Y,Z)||K(Y,Z)|\,dY\\
&\leq \frac{1}{4}\int_\R\langle Y\rangle^2|K(Y,Z)|^2\,dY+C(\sigma)\Big[\int_\R \langle Y\rangle^{-2}|\partial_YK(Y,Z)|^2\,dY+\langle Z\rangle|K(Z,Z)|\Big],
\end{split}
\end{equation}
which in view of \eqref{BKG11} implies that 
\begin{equation}\label{BKG13}
\begin{split}
&\int_\R \langle Y\rangle^2|K(Y,Z)|^2dY\lesssim_\sigma \langle Z\rangle|K(Z,Z)|,\\
&\alpha\int_\R |K(Y,Z)|^2\,dY+\int_\R |\partial_YK(Y,Z)|^2dY\lesssim_\sigma |K(Z,Z)|.
\end{split}
\end{equation}
Using the inequality that for any $f\in H^1(I)$ on an interval $I\subseteq \R$,
\begin{equation}\label{BKG14}
\|f\|_{L^{\infty}(I)}\lesssim \|f\|_{L^2(I)}|I|^{-1/2}+\|f'\|_{L^2(I)}|I|^{1/2}, 
\end{equation}
we obtain from \eqref{BKG13} that
\begin{equation}\label{BKG15}
\begin{split}
|K(Z,Z)|&\lesssim \|\partial_YK(\cdot,Z)\|_{L^2(|Y-Z|\lesssim \langle Z\rangle^{-1/2})}\langle Z\rangle^{-1/4}+ \|K(\cdot,Z)\|_{L^2(|Y-Z|\lesssim \langle Z\rangle^{-1/2})}\langle Z\rangle^{1/4}\\
&\lesssim_\sigma \langle Z\rangle^{-1/4}|K(Z,Z)|^{1/2},
\end{split}
\end{equation}
and if $\alpha\ge1$
\begin{equation}\label{BKG15'}
\begin{split}
|K(Z,Z)|&\lesssim \|\partial_YK(\cdot,Z)\|_{L^2(|Y-Z|\lesssim\langle\alpha\rangle^{-1/2})}\langle \alpha\rangle^{-1/4}+\|K(\cdot,Z)\|_{L^2(|Y-Z|\lesssim\langle\alpha\rangle^{-1/2})}\langle \alpha\rangle^{1/4}\\
&\lesssim_\sigma \alpha^{-1/4}|K(Z,Z)|^{1/2}.
\end{split}
\end{equation}
Therefore 
\begin{equation}\label{BKG16}
|K(Z,Z)|\lesssim_\sigma \langle \alpha,Z\rangle^{-1/2}.
\end{equation}
The desired bounds \eqref{BKG6}-\eqref{BKG7} follow from \eqref{BKG13} and \eqref{BKG16}. 

It remains to prove \eqref{BKG8}. Suppose that for some $\lambda$ with $0<\lambda\lesssim_\sigma1$,
\begin{equation}\label{BKG17}
|K(Z,Z)|=\lambda^2\langle \alpha, Z\rangle^{-1/2}.
\end{equation}
It follows from inequality \eqref{BKG13}, estimates similar to the bounds \eqref{BKG15}-\eqref{BKG15'}, and the equation \eqref{BKG2} that 
\begin{equation}\label{BKG18}
|K(Y,Z)|\lesssim_\sigma \lambda \langle \alpha, Z\rangle^{-1/2}, \quad{\rm for}\,\,|Y-Z|\leq \langle \alpha,Z\rangle^{-1/2},
\end{equation}
and 
\begin{equation}\label{BKG19}
|\partial_Y^2K(Y,Z)|\lesssim_\sigma \lambda \langle \alpha, Z\rangle^{1/2}, \quad{\rm for}\,\,0\neq|Y-Z|\leq \langle \alpha, Z\rangle^{-1/2}.
\end{equation}
Therefore
\begin{equation}\label{BKG20}
|\partial_YK(Y,Z)|\lesssim_\sigma \lambda\quad{\rm for}\,\,|Y-Z|\leq \langle \alpha, Z\rangle^{-1/2}.
\end{equation}
In view of \eqref{BKG2} and the bound \eqref{BKG20}, we conclude that
\begin{equation}\label{BKG21}
\lambda\approx_\sigma 1.
\end{equation}
The lemma is proved.

\end{proof}

\subsection{Local-to-global energy estimates}
Denote for $A\in\R$, 
\begin{equation}\label{BKG25}
\beta_+(A,Z):=\|K(\cdot,Z)\|_{L^2([A, A+\langle \alpha,A\rangle^{-1/2}])},\quad \beta_-(A,Z):=\|K(\cdot,Z)\|_{L^2([A-\langle \alpha,A\rangle^{-1/2},A])}.
\end{equation}
We show that $\beta_+(A,Z)$ controls $K(Y,Z)$ for $Y\in[A,+\infty)$ and $\beta_-(A,Z)$ controls $K(Y,Z)$ for $Y\in(-\infty,A]$. These local-to-global bounds play an important role in our argument below. 

\begin{lemma}\label{BKGL1}
For $A, Z\in\R$, we have the following bounds.

(i) For $A\ge Z+\langle\alpha,Z\rangle^{-1/2}$, we have the bounds
\begin{equation}\label{BKGL2}
\int_{A}^\infty \frac{\langle\alpha, Y\rangle^2}{\langle\alpha, A\rangle}|K(Y,Z)|^2+|\partial_YK(Y,Z)|^2\,dY\lesssim_\sigma\langle\alpha,A\rangle (\beta_+(A,Z))^2;
\end{equation}

(ii) For $A\leq Z-\langle\alpha,Z\rangle^{-1/2}$, we have the bounds 
\begin{equation}\label{BKGL3}
\int_{-\infty}^A \frac{\langle\alpha, Y\rangle^2}{\langle\alpha, A\rangle}|K(Y,Z)|^2+|\partial_YK(Y,Z)|^2\,dY\lesssim_\sigma\langle\alpha,A\rangle  (\beta_-(A,Z))^2.
\end{equation}

\end{lemma}

\begin{proof}
We focus on the proof of \eqref{BKGL2} and assume that $Z\leq0$, the case of $Z\ge0$ and the proof of \eqref{BKGL3} being similar. For $A\ge Z+\langle\alpha,Z\rangle^{-1/2}$, we choose a cutoff function $\varphi\in C_c^\infty((A,+\infty))$ such that $\varphi\equiv 1$ on $[A+\langle \alpha, A\rangle^{-1/2}, +\infty)$ and $|\varphi'|\lesssim \langle\alpha,A\rangle^{1/2}$. By multiplying $\varphi^2(Y)\overline{K(Y,Z)}$ and $\varphi^2(Y)V(Y)\overline{K(Y,Z)}$ to equation \eqref{BKG2}, and integrating over $\R$  we obtain that 
\begin{equation}\label{BKG22}
\begin{split}
&-\int_\R \varphi^2(Y)|\partial_YK(Y,Z)|^2\,dY-\alpha\int_\R\varphi^2(Y) |K(Y,Z)|^2dY\\
&-2\int_\R \varphi(Y)\varphi'(Y)\partial_YK(Y,Z) \overline{K(Y,Z)}\,dY+i\int_\R\varphi^2(Y) V(Y)|K(Y,Z)|^2\,dY=0,
\end{split}
\end{equation}
and
\begin{equation}\label{BKG23}
\begin{split}
&-\alpha\int_\R \varphi^2(Y)V(Y)|K(Y,Z)|^2dY-\int_\R \varphi^2(Y)V(Y)|\partial_YK(Y,Z)|^2dY\\
&-\int_\R\Big[2 \varphi(Y)\varphi'(Y)V(Y)+\varphi^2(Y)V'(Y)\Big]\partial_YK(Y,Z) \overline{K(Y,Z)}\,dY\\
&+i\int_\R \varphi^2(Y)V^2(Y)|K(Y,Z)|^2\,dY=0.
\end{split}
\end{equation}
We get from \eqref{BKG22}-\eqref{BKG23} that
\begin{equation}\label{BKG24}
\begin{split}
&\alpha\int_\R \varphi^2(Y)|K(Y,Z)|^2dY+\int_\R \varphi^2(Y)|\partial_YK(Y,Z)|^2\,dY\\
&\leq 2\int_\R \varphi(Y)\big|\varphi'(Y)||\partial_YK(Y,Z) \overline{K(Y,Z)}|\,dY,\\
&\int_\R V^2(Y)\varphi^2(Y)|K(Y,Z)|^2\,dY\\
&\leq \int_\R \Big[2|\varphi(Y)\varphi'(Y)V(Y)|+\varphi^2(Y)|V'(Y)|\Big]\big|\partial_YK(Y,Z) \overline{K(Y,Z)}\big|\,dY.
\end{split}
\end{equation}
It follows from equation \eqref{BKG2} and the definition \eqref{BKG25} that
\begin{equation}\label{BKG26}
\|\partial_Y^2K(Y,Z)\|_{L^2(Y\in[A, A+\langle \alpha,A\rangle^{-1/2}])}\lesssim_\sigma \langle\alpha,A\rangle\beta_+(A,Z),
\end{equation}
and thus
\begin{equation}\label{BKG27}
\|\partial_YK(Y,Z)\|_{L^2(Y\in[A, A+\langle \alpha,A\rangle^{-1/2}])}\lesssim_\sigma \langle\alpha,A\rangle^{1/2}\beta_+(A,Z).
\end{equation}
We then conclude from \eqref{BKG24} that
\begin{equation}\label{BKG28}
\alpha\int_\R\varphi^2(Y) |K(Y,Z)|^2dY+\int_\R \varphi^2(Y)|\partial_YK(Y,Z)|^2\,dY\lesssim_\sigma \langle\alpha,A\rangle(\beta_+(A,Z))^2,
\end{equation}
and from \eqref{BKG24}, \eqref{BKG28}, and using also Poincar\'e inequality that
\begin{equation}\label{BKG29}
\begin{split}
&\int_\R \varphi^2(Y)\langle Y\rangle^2|K(Y,Z)|^2\,dY\\
&\leq \frac{1}{4}\int_\R \varphi^2(Y)\langle Y\rangle^2|K(Y,Z)|^2dY+C(\sigma)\Big[\int_\R\frac{\varphi^2(Y)}{\langle Y\rangle^2}|\partial_YK(Y,Z)|^2dY+\langle\alpha,A\rangle^2\big(\beta_+(A,Z)\big)^2\Big],
\end{split}
\end{equation}
which implies that
\begin{equation}\label{BKG30}
\int_\R\varphi^2(Y) \langle Y\rangle^2|K(Y,Z)|^2\,dY\lesssim_\sigma \langle\alpha,A\rangle^2\big(\beta_+(A,Z)\big)^2.
\end{equation}
The desired bounds \eqref{BKGL2} follow from \eqref{BKG28} and \eqref{BKG30}.
\end{proof}

\subsection{Entanglement inequalities}
The following ``entanglement inequalities" play an important role in our argument.
\begin{lemma}\label{BKGL31}
Assume that $A_1, A_2\in\R$ satisfy the conditions that $A_1\leq A_2$, and either $[A_1, A_2]\subseteq [0,\infty)$ or $[A_1, A_2]\subseteq (-\infty,0]$. Fix a nonnegative cutoff function $\varphi\in C^1([A_1,A_2])$ such that $\varphi(A_1)=\varphi(A_2)=0$. We extend the domain of $\varphi$ to $\R$ by setting $\varphi(A)=0$ for $A\not\in[A_1,A_2]$. Then we have the following bounds for some $c_0:=c_0(\sigma)\in(0,1)$,
\begin{equation}\label{BKGL32}
\int_\R\Big[|\varphi'(Y)|^2-c^2_0\,\langle\alpha, Y\rangle|\varphi(Y)|^2\Big]|K(Y,Z)|^2\,dY\ge0.
\end{equation}
\end{lemma}

\begin{remark}
In applications below, we shall choose the cutoff function $\varphi$ so that $|\varphi'(Y)|^2-c_0^2\,\langle\alpha,Y\rangle|\varphi(Y)|^2=0$ on $[A_1, A_2]$ except when $$Y=A_1+O(\langle\alpha, A_1\rangle^{-1/2})\quad{\rm or}\quad Y=A_2+O(\langle\alpha, A_2\rangle^{-1/2}).$$ Then \eqref{BKGL32} implies the control of $K(Y, Z)$ near $A_1$ by $K(Y,Z)$ around $A_2$, and vice versa. Since $[A_1,A_2]$ is a possibly very large interval, there is an ``entanglement" in the behavior of $K(Y,Z)$ near the two end points $Y=A_1$ and $Y=A_2$, which inspires our choice of the terminology.
\end{remark}

\begin{proof}
Using energy estimates similar to \eqref{BKG22} and noting that $V$ does not change sign on $[A_1,A_2]$, we obtain that 
\begin{equation}\label{BKG31}
\begin{split}
&\alpha\int_\R\varphi^2(Y)|K(Y,Z)|^2dY+\int_\R \varphi^2(Y)|\partial_YK(Y,Z)|^2\,dY+\Big|\int_\R \varphi^2(Y)V(Y)|K(Y,Z)|^2\,dY\Big|\\
&\lesssim\int_\R \varphi(Y)\big|\varphi'(Y)\partial_YK(Y,Z) \overline{K(Y,Z)}\big|\,dY.
\end{split}
\end{equation}
Therefore,
\begin{equation}\label{BKG32}
\begin{split}
&\alpha\int_\R\varphi^2(Y)|K(Y,Z)|^2dY+\int_\R \varphi^2(Y)\langle Y\rangle|K(Y,Z)|^2\,dY+\int_\R \varphi^2(Y)|\partial_YK(Y,Z)|^2\,dY\\
&\lesssim_\sigma\int_\R |\varphi'(Y)|^2|K(Y,Z)|^2\,dY.
\end{split}
\end{equation}
It follows from \eqref{BKG32} that for some $c_0(\sigma)\in(0,1)$,
\begin{equation}\label{BKG33}
\int_\R\Big[|\varphi'(Y)|^2-c^2_0\,\langle\alpha, Y\rangle|\varphi(Y)|^2\Big]|K(Y,Z)|^2\,dY\ge0.
\end{equation}
\end{proof}

\subsection{Proof of \eqref{BKG3}-\eqref{BKG4} for $\langle \alpha\rangle^{-1/2}\ll_\sigma |Z|$} We now turn to the proof of \eqref{BKG3}-\eqref{BKG4} in the case $\langle \alpha\rangle^{-1/2}\ll_\sigma |Z|$. We assume, without loss of generality, that $Z\ge0$ and consider several cases. 

{\it Case I: $|Y-Z|\lesssim_\sigma \langle\alpha, Z\rangle^{-1/2}$.} The desired bounds \eqref{BKG3}-\eqref{BKG4} follow from the bounds \eqref{BKG7}-\eqref{BKG8} and equation \eqref{BKG2}, if $c_0(\sigma)\in(0,1)$ is chosen sufficiently small.  

{\it Case II: $\langle \alpha\rangle^{-1/2}\ll_\sigma Y<Z-10\langle \alpha, Z\rangle^{-1/2}$.} We shall use \eqref{BKGL32} and need to give a more precise description on the function $\varphi$. Fix $1\ll_\sigma\ell$ to be determined below. Let $A_1:=Y-\ell\langle\alpha,Y\rangle^{-1/2}$, $A_2:=Z-\langle\alpha,Z\rangle^{-1/2}$. We choose $\varphi$ so that 
\begin{equation}\label{BKG34}
\begin{split}
&\varphi'(X)=-c_0\langle\alpha, X\rangle^{1/2}\varphi,\quad {\rm for}\,\,X\in\big[Y+\langle\alpha, Y\rangle^{-1/2},Z-2\langle\alpha,Z\rangle^{-1/2}\big],\\
&\varphi(Y-\ell\langle\alpha, Y\rangle^{-1/2})=\varphi(Z-\langle\alpha,Z\rangle^{-1/2})=0,\quad \varphi(Z-2\langle\alpha,Z\rangle^{-1/2})=1,\\
&\varphi''(X)\equiv0,\quad{\rm for}\,\,X\in\big[Y-\ell\langle\alpha, Y\rangle^{-1/2}, Y+ \langle\alpha, Y\rangle^{-1/2}\big]\cup \big[Z-2\langle\alpha,Z\rangle^{-1/2},Z-\langle\alpha,Z\rangle^{-1/2}\big].
\end{split}
\end{equation}
It follows from \eqref{BKG34} that for some $c_1=c_1(\sigma)\in(0,1)$
\begin{equation}\label{BKG34.1}
m_\ast:=\varphi(Y+\langle\alpha, Y\rangle^{-1/2})\gtrsim_\sigma e^{c_1\langle\alpha,Z\rangle^{1/2}|Y-Z|}.
\end{equation}
Moreover, assuming that $Y\gg \ell \langle\alpha\rangle^{1/2}$, from straightforward computation we can conclude that for a suitable $c_2=c_2(\sigma)\in(0,1)$,
\begin{equation}\label{BKG34.2}
c_0\,\langle\alpha, X\rangle|\varphi(X)|^2-|\varphi'(X)|^2\ge0,\quad {\rm for}\,\,X\in\big[ Y-(\ell-c_2^{-1/2})\langle\alpha, Y\rangle^{-1/2},Y+\langle\alpha, Y\rangle^{-1/2}\big],
\end{equation}
and
\begin{equation}\label{BKG34.3}
c_0\,\langle\alpha, X\rangle|\varphi(X)|^2-|\varphi'(X)|^2\gtrsim c_2 \langle\alpha, Y\rangle m_\ast,\quad {\rm for}\,\,X\in\big[ Y-(\ell/2)\langle\alpha, Y\rangle^{-1/2},Y+\langle\alpha, Y\rangle^{-1/2}\big].
\end{equation}
Therefore, assuming that $Y\gg\ell\langle\alpha\rangle^{-1/2}$, we see from \eqref{BKGL32} and \eqref{BKG34.1}-\eqref{BKG34.3} that
\begin{equation}\label{BKG36}
\begin{split}
&m_\ast\int_{Y-(\ell/2)\langle\alpha, Y\rangle^{-1/2}}^{Y+\langle\alpha, Y\rangle^{-1/2}}\langle\alpha, Y\rangle|K(X,Z)|^2\,dX-C(\sigma)m_\ast\int_{Y-\ell\langle\alpha, Y\rangle^{-1/2}}^{Y-(\ell-c_1)\langle\alpha, Y\rangle^{-1/2}}\langle\alpha, Y\rangle|K(X,Z)|^2\,dX\\
&\lesssim\int_{Y-\ell\langle\alpha, Y\rangle^{-1/2}}^{Y+ \langle\alpha, Y\rangle^{-1/2}}\big(c_0\langle \alpha,X\rangle|\varphi(X)|^2-|\varphi'(X)|^2\big)|K(X,Z)|^2\,dX\\
&\lesssim_\sigma\int_{Z-2\langle\alpha,Z\rangle^{-1/2}}^{Z-\langle\alpha,Z\rangle^{-1/2}}\big(\langle X\rangle|\varphi(X)|^2+|\varphi'(X)|^2\big)|K(X,Z)|^2\,dX\lesssim_\sigma \langle\alpha,Z\rangle^{-1/2}.
\end{split}
\end{equation}
If we choose $\ell\ge1$ to be sufficiently large depending on $\sigma$, so that by the local-to-global energy bounds \eqref{BKGL3}, we have
\begin{equation}\label{BKG36.5}
\int_{Y-(\ell/2)\langle\alpha, Y\rangle^{-1/2}}^{Y+\langle\alpha, Y\rangle^{-1/2}}\langle\alpha, Y\rangle|K(X,Z)|^2\,dX\ge 2C(\sigma)\int_{Y-\ell\langle\alpha, Y\rangle^{-1/2}}^{Y-(\ell-c_1)\langle\alpha, Y\rangle^{-1/2}}\langle\alpha, Y\rangle|K(X,Z)|^2\,dX,
\end{equation}
 and therefore, by \eqref{BKG36}, for a suitable $c_3(\sigma)\in(0,1)$,
\begin{equation}\label{BKG37}
\int_{Y-\langle\alpha, Y\rangle^{-1/2}}^{Y+ \langle\alpha, Y\rangle^{-1/2}}|K(X,Z)|^2\,dX\lesssim_\sigma \frac{1}{\langle\alpha,Y\rangle}\langle\alpha,Z\rangle^{-1/2}e^{-c_3\langle\alpha, Z\rangle^{1/2}|Y-Z|}.
\end{equation}
The desired bounds \eqref{BKG3}-\eqref{BKG4} then follow from \eqref{BKG37}, using the equation \eqref{BKG2}. 

{\it Case III: $|Y|\lesssim_\sigma \langle\alpha\rangle^{-1/2}$.} The desired bounds \eqref{BKG3}-\eqref{BKG4} follow from {\it Case II} and the local-to-global energy bounds \eqref{BKGL3}, using the equation \eqref{BKG2}.

{\it Case IV: $Y<0$ and $\langle\alpha\rangle^{-1/2}\ll_\sigma |Y|$.} This case is similar to {\it Case II}, using the entanglement inequality \eqref{BKGL32} by choosing the cutoff function $\varphi$ so that (with sufficiently large $\ell\gg_\sigma1$ and $A_1:=Y-\ell\langle\alpha,Y\rangle^{-1/2}$, $A_2:=-\langle\alpha\rangle^{-1/2}$)
\begin{equation}\label{BKG38}
\begin{split}
&\varphi'(X)=-c_0\langle\alpha, X\rangle^{1/2}\varphi,\quad {\rm for}\,\,X\in\big[Y+\langle\alpha, Y\rangle^{-1/2},-2\langle\alpha\rangle^{-1/2}\big],\\
&\varphi(Y-\ell\langle\alpha, Y\rangle^{-1/2})=\varphi(-\langle\alpha\rangle^{-1/2})=0,\quad \varphi(-2\langle\alpha\rangle^{-1/2})=1,\\
&\varphi''(X)\equiv0,\quad{\rm for}\,\,X\in\big[Y-\ell\langle\alpha, Y\rangle^{-1/2}, Y+ \langle\alpha, Y\rangle^{-1/2}\big]\cup \big[-2\langle\alpha\rangle^{-1/2},-\langle\alpha\rangle^{-1/2}\big].
\end{split}
\end{equation}

{\it Case V: $\langle\alpha, Z\rangle^{-1/2}\ll_\sigma Y-Z$.} This case is also similar to {\it Case II}, using the entanglement inequality \eqref{BKGL32} by choosing the cutoff function $\varphi$ so that (with sufficiently large $\ell\gg1$ and $A_1:=Z+\langle\alpha,Z\rangle^{-1/2}$, $A_2:=Y+\ell\langle\alpha,Y\rangle^{-1/2}$)
\begin{equation}\label{BKG39}
\begin{split}
&\varphi'(X)=c_0\langle\alpha, X\rangle^{1/2}\varphi,\quad {\rm for}\,\,X\in\big[Z+2\langle\alpha,Z\rangle^{-1/2}, Y-\langle\alpha, Y\rangle^{-1/2}\big],\\
&\varphi(Y+\ell\langle\alpha, Y\rangle^{-1/2})=\varphi(Z+\langle\alpha,Z\rangle^{-1/2})=0,\quad \varphi(Z+2\langle\alpha,Z\rangle^{-1/2})=1,\\
&\varphi''(X)\equiv0,\quad{\rm for}\,\,X\in\big[Y-\langle\alpha, Y\rangle^{-1/2}, Y+ \ell\langle\alpha, Y\rangle^{-1/2}\big]\cup \big[Z+\langle\alpha,Z\rangle^{-1/2},Z+2\langle\alpha,Z\rangle^{-1/2}\big].
\end{split}
\end{equation}

Summarizing {\it Case I} to {\it Case V}, the proof of \eqref{BKG3}-\eqref{BKG4} is then complete, assuming that $\langle\alpha\rangle^{-1/2}\ll_\sigma|Z|$. 

\subsection{Proof of \eqref{BKG3}-\eqref{BKG4} for $|Z|\lesssim_\sigma\langle \alpha\rangle^{-1/2}$}  We note that the case $|Z|\lesssim_\sigma \langle\alpha\rangle^{-1/2}$ is easier, by treating the cases $|Y|\lesssim_\sigma \langle\alpha\rangle^{-1/2}$, $\langle\alpha\rangle^{-1/2}\ll_\sigma Y$ and $\langle\alpha\rangle^{-1/2}\ll_\sigma-Y$ separately.

%
%

\section{High Reynolds number Gevrey bounds for the generalized Airy operator}\label{hrn}
In this section, we prove precise estimates on the kernel of the generalized Airy operator in Gevrey spaces, in the high Reynolds number regime. To obtain the refined regularity structure in the study of the main Orr-Sommerfeld equation \eqref{intP8} in $v,w$ variables, for $0<|\epsilon|<1/9$ we define $k_\epsilon(v,\rho;w)$ as the solution to the equation for $v,\rho, w\in\R$ that
\begin{equation}\label{GKA6}
\epsilon\partial_v^2k_\epsilon(v,\rho;w)+\epsilon\frac{\partial_vB^\ast(v+w)}{B^\ast(v+w)}\partial_vk_\epsilon(v,\rho;w)-i\frac{v}{(B^\ast(v+w))^2}k_\epsilon(v,\rho;w)=(\rho+i\epsilon^{1/3})\delta(v-\rho).
\end{equation}
Fix a Gevrey smooth function $\Psi$ satisfying
\begin{equation}\label{GKA10}
\Psi\in C_0^\infty(-4,4),\quad \Psi\equiv 1\,\, {\rm on}\,\, [-2,2], \quad\sup_{\xi\in\R}\Big[e^{\langle \xi\rangle^{3/4}}\widehat{\,\Psi\,}(\xi)\Big]\lesssim1. 
\end{equation}
and define the Fourier multiplier operator $A$ for any $f\in L^2(\R)$, as
\begin{equation}\label{GKA10.1}
\widehat{\,Af\,}(\xi):=e^{\delta_0\langle \xi\rangle^{1/2}}, \qquad {\rm for}\,\,\xi\in\R. 
\end{equation}
Our main goal of the section is to prove the following bounds on the kernel $k_\epsilon(v,\rho;w)$. 

 \begin{proposition}\label{GKA11}
 Assume that $\epsilon\in(-1/8,1/8)\backslash\{0\}$ and $k\in\Z\backslash\{0\}$. Let $w_0\in\R$ and $\Psi$ be as in \eqref{GKA10}. Denote for $v,\rho,w\in\R$,
 \begin{equation}\label{GKA12}
 k_{\epsilon,w_0}(v,\rho;w):=\Psi(w-w_0)k_\epsilon(v,\rho;w).
 \end{equation}
 Then for suitable $\delta_0\in(0,1)$, we have the following bounds.
 
 (i) For all $v,\rho\in\R$, 
 \begin{equation}\label{GKA13}
\sup_{\xi\in\R} \Big|e^{\delta_0\langle\xi\rangle^{1/2}}\widehat{\,\,k_{\epsilon,w_0}}(v,\rho;\cdot)(\xi)\Big|\lesssim_\sigma |\epsilon|^{-1/3}\langle \epsilon^{-1/3}\rho\rangle^{1/2}e^{-\delta_0\langle\epsilon^{-1/3}v,\epsilon^{-1/3}\rho\rangle^{1/2}\epsilon^{-1/3}|v-\rho|},
 \end{equation}
 and
  \begin{equation}\label{GKA14}
 \sup_{\xi\in\R} \Big|e^{\delta_0\langle\xi\rangle^{1/2}}\partial_v\widehat{\,\,k_{\epsilon,w_0}}(v,\rho;\cdot)(\xi)\Big|\lesssim_\sigma |\epsilon|^{-2/3}\langle \epsilon^{-1/3}\rho\rangle e^{-\delta_0\langle\epsilon^{-1/3}v,\epsilon^{-1/3}\rho\rangle^{1/2}\epsilon^{-1/3}|v-\rho|},
 \end{equation}
 
 (ii) In addition, we have the decomposition for $v,\rho,w\in\R$,
 \begin{equation}\label{GKA14.1}
 \Psi(\rho)k_{\epsilon}(v,\rho;w):=k^1_\epsilon(v,\rho;w)+k^2_\epsilon(v,\rho;w),
 \end{equation}
 where the kernels $k^j_\epsilon(v,\rho;w), j\in\{1,2\}$ are supported in $\rho\in[-4,4]$, and satisfy the bounds 
 \begin{equation}\label{GKA14.2}
 \begin{split}
& \sup_{\xi\in\R}\big|e^{\delta_0\langle \xi\rangle^{1/2}}\widehat{\,\,k^1_{\epsilon,w_0}}(v,\rho;\xi)\big|\lesssim \langle\epsilon^{-1/3}v\rangle^{3/2} e^{-\delta_0\langle\epsilon^{-1/3}v,\epsilon^{-1/3}\rho\rangle^{1/2}\epsilon^{-1/3}|v-\rho|},\\
& \sup_{\xi\in\R}\big|e^{\delta_0\langle \xi\rangle^{1/2}}\partial_v\widehat{\,\,k^1_{\epsilon,w_0}}(v,\rho;\xi)\big|\lesssim |\epsilon|^{-1/3}\langle\epsilon^{-1/3}v\rangle^2 e^{-\delta_0\langle \xi\rangle^{1/2}}e^{-\delta_0\langle\epsilon^{-1/3}v,\epsilon^{-1/3}\rho\rangle^{1/2}\epsilon^{-1/3}|v-\rho|},
 \end{split}
 \end{equation}
 and
  \begin{equation}\label{GKA14.3}
\sup_{\alpha\in\R}\int_\R\sup_{\eta\in\R}\Big[\big| \widehat{\,\,k_{\epsilon,w_0}^2}(\alpha,\beta;\eta)\big|e^{\delta_0\langle\eta\rangle^{1/2}}\Big] d\beta\lesssim1.
 \end{equation}

 \end{proposition}
The bounds \eqref{GKA13}-\eqref{GKA14} are sufficient for applications when $|v|+|\rho|\gtrsim1$, while the bounds \eqref{GKA14.1}-\eqref{GKA14.2} are useful for $|v|+|\rho|\lesssim1$. The rest of the section is devoted to the proof of Proposition \ref{GKA11}, which we organize into subsections. 

\subsection{Preliminary pointwise bounds}
We begin with the following pointwise bounds, as a corollary of Proposition \ref{BKG1}.
\begin{lemma}\label{GKA1}
Assume that $y_0\in\R, \alpha\in\R$ and $0<|\epsilon|<1/8$ with $\epsilon\alpha\ge0$. Let $k_{\epsilon,\alpha}^\ast(y,z;y_0)$ be the kernel for the operator $\epsilon\partial_y^2+i(b(y)-b(y_0))$ on $\R$. More precisely, for $y, z\in\R$, 
\begin{equation}\label{GKA1.1}
(\epsilon\partial_y^2-\alpha)k_{\epsilon,\alpha}^\ast(y,z;y_0)+i(b(y_0)-b(y))k_{\epsilon,\alpha}^\ast(y,z;y_0)=\delta(y-z).
\end{equation}
 Then for a suitable $c_0\in(0,1)$, we have the bound  for $y,z\in\R$, 
\begin{equation}\label{GKA2}
|k_{\epsilon,\alpha}^\ast(y,z;y_0)|\lesssim \frac{\epsilon^{-2/3}}{\langle\epsilon^{-1/3}(z-y_0),\epsilon^{-1/3}\alpha\rangle^{1/2}}e^{-c_0\langle \epsilon^{-1/3}(y-y_0), \epsilon^{-1/3}(z-y_0),\epsilon^{-1/3}\alpha\rangle^{1/2}\epsilon^{-1/3}|y-z|}
\end{equation}
and 
\begin{equation}\label{GKA3}
|\partial_yk_{\epsilon,\alpha}^\ast(y,z;y_0)|\lesssim \epsilon^{-1}e^{-c_0\langle \epsilon^{-1/3}(y-y_0), \epsilon^{-1/3}(z-y_0),\epsilon^{-1/3}\alpha\rangle^{1/2}\epsilon^{-1/3}|y-z|}
.
\end{equation}

\end{lemma}

\begin{proof}
The proof follows from Proposition \ref{BKG1} and the change of variables for $y,z,y_0\in\R$,
\begin{equation}\label{GKA3.1}
y-y_0:=\epsilon^{1/3}Y,\quad z-y_0:=\epsilon^{1/3}Z, \quad b(y)-b(y_0):=\epsilon^{1/3}V(Y),\quad k_\epsilon^\ast(y,z;y_0):=K(Y,Z).
\end{equation}

\end{proof}


We have the following preliminary pointwise estimates, as a consequence of Lemma \ref{GKA1} and the change of variables \eqref{intP7}.
\begin{lemma}\label{GKA7}
Assume that $\epsilon\in(-1/8,1/8)\backslash\{0\}$ and $k\in\Z\backslash\{0\}$. Let $k_\epsilon(v,\rho;w)$ be defined as in \eqref{GKA6}. For suitable $c_0=c_0(\sigma)\in(0,1)$, we have the bounds for all $v,\rho,w\in\R$,
\begin{equation}\label{GKA8}
|k_\epsilon(v,\rho;w)|\lesssim|\epsilon|^{-1/3}\langle \epsilon^{-1/3}\rho\rangle^{1/2}e^{-c_0\langle\epsilon^{-1/3}v,\epsilon^{-1/3}\rho\rangle^{1/2}\epsilon^{-1/3}|v-\rho|},
\end{equation}
and 
\begin{equation}\label{GKA9}
|\partial_vk_\epsilon(v,\rho;w)|\lesssim  |\epsilon|^{-2/3}\langle \epsilon^{-1/3}\rho\rangle e^{-c_0\langle\epsilon^{-1/3}v,\epsilon^{-1/3}\rho\rangle^{1/2}\epsilon^{-1/3}|v-\rho|}.
\end{equation}

\end{lemma}

\begin{proof}
We assume without loss of generality, $\epsilon>0$. Define, using the change of variables \eqref{intP7}, for $v,\rho,w\in\R$, with $\alpha=0$,
\begin{equation}\label{GKA9.1}
k^\dagger_\epsilon(v,\rho;w)=k_{\epsilon,\alpha}^\ast(y,z;y_0).
\end{equation}
Then $k^\dagger_\epsilon(v,\rho;w)$ satisfies for $v,\rho,w\in\R$,
\begin{equation}\label{GKA9.2}
\Big[\epsilon(B^\ast(v))^2\partial_v^2+\epsilon B^\ast\partial_vB^\ast(v)\partial_v+i(v-w)\Big]k^\dagger_\epsilon(v,\rho;w)=B^\ast(\rho)\delta(v-\rho).
\end{equation}
By choosing $c_0$ as a fraction of the $c_0$ appearing in \eqref{GKA2}-\eqref{GKA3}, the desired bounds follow from Lemma \ref{GKA1} and the identity that for $v,\rho,w\in\R$,
\begin{equation}\label{GKA9.3}
k_\epsilon(v,\rho;w)=(\rho+i\epsilon^{1/3})B^\ast(\rho+w)k^\dagger_\epsilon(v+w,\rho+w;w).
\end{equation}

\end{proof}

 \subsection{Proof of \eqref{GKA13}-\eqref{GKA14}}
 We note that for $v, \rho\in\R$, 
 \begin{equation}\label{GKA15}
 \max\{|v|^{1/2}, |\rho|^{1/2}\}|v-\rho|\approx \big||v|^{1/2}v-|\rho|^{1/2}\rho\big|.
 \end{equation}
 We shall, for the convenience of argument, prove the following bounds, which imply \eqref{GKA13}-\eqref{GKA14} up to the constant $\delta_0\in(0,1)$,
 \begin{equation}\label{GKA15}
 \left\|A\Big[k_{\epsilon,w_0}(v,\rho;\cdot)\Big](w)\right\|_{L^2(\R)}\lesssim \epsilon^{-1/3}\langle \epsilon^{-1/3}\rho\rangle^{1/2} e^{-\delta_0\epsilon^{-1/2}||v|^{1/2}v-|\rho|^{1/2}\rho|},
 \end{equation}
 and
  \begin{equation}\label{GKA16}
 \left\|\partial_vA\Big[k_{\epsilon,w_0}(v,\rho;\cdot)\Big](w)\right\|_{L^2(\R)}\lesssim \epsilon^{-2/3}\langle \epsilon^{-1/3}\rho\rangle e^{-\delta_0\epsilon^{-1/2}||v|^{1/2}v-|\rho|^{1/2}\rho|}.
 \end{equation}
Multiplying $\Psi(w-w_0)$ to the equation \eqref{GKA6} and applying the Fourier multiplier $A$ (acting on the variable $w$), we obtain that 
 \begin{equation}\label{GKA17}
 \begin{split}
&\bigg[\epsilon\partial_v^2+\epsilon\frac{\partial_vB^\ast(v+w)}{B^\ast(v+w)}\partial_v-\frac{iv}{(B^\ast(v+w))^2}\bigg]A\big[k_{\epsilon,w_0}(v,\rho;\cdot)\big](w)\\
&=(\rho+i\epsilon^{1/3})\delta(v-\rho)A\big[\Psi(\cdot-w_0)\big](w)+\mathcal{C}_{\epsilon,w_0}^1(v,\rho;w)+\mathcal{C}_{\epsilon,w_0}^2(v,\rho;w),
\end{split}
\end{equation}
where the commutator terms $\mathcal{C}_{\epsilon,w_0}^1$ and $\mathcal{C}_{\epsilon,w_0}^2$ are given by 
\begin{equation}\label{GKA18}
\mathcal{C}_{\epsilon,w_0}^1(v,\rho;w):=\epsilon\bigg\{\frac{\partial_vB^\ast(v+w)}{B^\ast(v+w)}\partial_vA\big[k_{\epsilon,w_0}(v,\rho;\cdot)\big](w)-A\bigg[\frac{\partial_vB^\ast(v+\cdot)}{B^\ast(v+\cdot)}\partial_vk_{\epsilon,w_0}(v,\rho;\cdot)\bigg](w)\bigg\},
\end{equation}
 and 
 \begin{equation}\label{GKA19}
 \mathcal{C}_{\epsilon,w_0}^2(v,\rho;w):=-iv\bigg\{\frac{1}{(B^\ast(v+w))^2}A\big[k_{\epsilon,w_0}(v,\rho;\cdot)\big](w)-A\bigg[\frac{1}{(B^\ast(v+\cdot))^2}k_{\epsilon,w_0}(v,\rho;\cdot)\bigg](w)\bigg\}.
 \end{equation} 
 Fix Gevrey cutoff functions $\Psi^\ast$ and $\Psi^\dagger$ satisfying
 \begin{equation}\label{GKA19.1}
 \begin{split}
&\Psi^\ast\in C_0^\infty[-5,5],\quad \Psi^\ast\equiv 1\,\, {\rm on}\,\, [-4,4], \quad\sup_{\xi\in\R}\Big[e^{\langle \xi\rangle^{3/4}}\widehat{\,\Psi^\ast\,}(\xi)\Big]\lesssim1,\\
&\Psi^\dagger\in C_0^\infty[-6,6],\quad \Psi^\dagger\equiv 1\,\, {\rm on}\,\, [-5,5], \quad\sup_{\xi\in\R}\Big[e^{\langle \xi\rangle^{3/4}}\widehat{\,\Psi^\dagger\,}(\xi)\Big]\lesssim1.
\end{split}
\end{equation}

It follows from equation \eqref{GKA17} and the definition \eqref{GKA6} that for $v,\rho,w\in\R$,
\begin{equation}\label{GKA20}
\begin{split}
\big|\Psi^\ast(w-w_0)A\big[k_{\epsilon,w_0}(v,\rho;\cdot)\big](w)\big|\lesssim& \big|\Psi^\ast(w-w_0)A\big[\Psi(\cdot-w_0)\big](w)k_{\epsilon}(v,\rho;w)\big|\\
&+\sum_{j\in\{1,2\}}\bigg|\int_\R k_{\epsilon}(v,v';w)\Psi^\ast(w-w_0)\frac{\mathcal{C}^j_{\epsilon,w_0}(v',\rho;w)}{v'+i\epsilon^{1/3}}dv'\bigg|,
\end{split}
\end{equation}
and 
\begin{equation}\label{GKA21}
\begin{split}
\big|\Psi^\ast(w-w_0)\partial_vA\big[k_{\epsilon,w_0}(v,\rho;\cdot)\big](w)\big|\lesssim& \big|\Psi^\ast(w-w_0)A\big[\Psi(\cdot-w_0)\big](w)\partial_vk_{\epsilon}(v,\rho;w)\big|\\
&+\sum_{j\in\{1,2\}}\bigg|\int_\R \partial_vk_{\epsilon}(v,v';w)\Psi^\ast(w-w_0)\frac{\mathcal{C}^j_{\epsilon,w_0}(v',\rho;w)}{v'+i\epsilon^{1/3}}dv'\bigg|.
\end{split}
\end{equation}
Denote for $v,\rho\in\R$,
\begin{equation}\label{GKA22}
\begin{split}
M_\epsilon(v,\rho):=\sup_{w_0\in\R}\bigg\{&\epsilon^{1/3}\langle \epsilon^{-1/3}v\rangle^{-1/2}e^{\delta_0 \epsilon^{-1/2}||v|^{1/2}-|\rho|^{1/2}\rho|}\big\|A\big[k_{\epsilon,w_0}(v,\rho;\cdot)\big](w)\big\|_{L^2(w\in\R)}\\
&+\epsilon^{2/3}\langle \epsilon^{-1/3}v\rangle^{-1}e^{\delta_0 \epsilon^{-1/2}||v|^{1/2}-|\rho|^{1/2}\rho|}\big\|\partial_vA\big[k_{\epsilon,w_0}(v,\rho;\cdot)\big](w)\big\|_{L^2(w\in\R)}\bigg\}.
\end{split}
\end{equation}
In the above, $\delta_0$ can be chosen as a fraction of the $c_0$ appearing in \eqref{GKA8}-\eqref{GKA9}. Noting that for $v,\rho,w\in\R$,
\begin{equation}\label{GAK22.1}
\Psi^\dagger(w-w_0)k_{\epsilon,w_0}(v,\rho;w)=k_{\epsilon,w_0}(v,\rho;w),
\end{equation}
we have the following bounds on the commutator terms $ \mathcal{C}^1_{\epsilon,w_0}(v,\rho;w)$ and $\mathcal{C}^1_{\epsilon,w_0}(v,\rho;w)$,
\begin{equation}\label{GKA23}
\begin{split}
\|\Psi^\ast(w-w_0)\mathcal{C}^1_{\epsilon,w_0}(v,\rho;w)\|_{L^2(w\in\R)}&\lesssim \epsilon\Big\|\langle\partial_w\rangle^{-1/2}\partial_vA\big[k_{\epsilon,w_0}(v,\rho;\cdot)\big](w)\Big\|_{L^2(w\in\R)}\\
&\lesssim(\gamma M_\epsilon(v,\rho)+C_{\gamma}) \epsilon^{1/3}\langle \epsilon^{-1/3}v\rangle e^{-\delta_0\epsilon^{-1/2}||v|^{1/2}v-|\rho|^{1/2}\rho|},
\end{split}
\end{equation}
and 
\begin{equation}\label{GKA24}
\begin{split}
\|\Psi^\ast(w-w_0)\mathcal{C}^2_{\epsilon,w_0}(v,\rho;w)\|_{L^2(w\in\R)}&\lesssim |v|\Big\|\langle\partial_w\rangle^{-1/2}A\big[k_{\epsilon,w_0}(v,\rho;\cdot)\big](w)\Big\|_{L^2(w\in\R)}\\
&\lesssim(\gamma M_\epsilon(v,\rho)+C_{\gamma}) \langle \epsilon^{-1/3}v\rangle^{3/2} e^{-\delta_0\epsilon^{-1/2}||v|^{1/2}v-|\rho|^{1/2}\rho|},
\end{split}
\end{equation}
for any $\gamma\in(0,1)$ and suitable constant $C_\gamma\in(0,\infty)$. Using the bounds \eqref{GKA20}-\eqref{GKA21}, we obtain that for all $v,\rho\in\R$
\begin{equation}\label{GKA25}
M_\epsilon(v,\rho)\lesssim \gamma M_\epsilon(v,\rho)+C_\gamma.
\end{equation}
Choosing $\gamma$ sufficiently small, the desired bounds \eqref{GKA13}-\eqref{GKA14} then follow from \eqref{GKA23}-\eqref{GKA24} and Lemma \ref{pre300}. 
 
 \subsection{Proof of \eqref{GKA14.1}-\eqref{GKA14.3}}
 The bounds \eqref{GKA13}-\eqref{GKA14} are sufficient for our purposes for $|v|+|\rho|\ge1$. However we need more precise characterizations of the kernel $k_\epsilon(v,\rho;w)$ for $|v|+|\rho|\lesssim1$. The equation \eqref{GKA6} can be reformulated, for $v,\rho\in\R$, as
 \begin{equation}\label{GKA26}
 \begin{split}
 &\bigg[\epsilon\partial_v^2+\epsilon\frac{\partial_vB^\ast(w)}{B^\ast(w)}\partial_v-i\frac{v}{(B^\ast(w))^2}\bigg]k_\epsilon(v,\rho;w)=\epsilon\bigg[\frac{\partial_vB^\ast(w)}{B^\ast(w)}-\frac{\partial_vB^\ast(v+w)}{B^\ast(v+w)}\bigg]\partial_vk_\epsilon(v,\rho;w)\\
 &-iv\bigg[\frac{1}{(B^\ast(w))^2}-\frac{1}{(B^\ast(v+w))^2}\bigg]k_\epsilon(v,\rho;w)+(\rho+i\epsilon^{1/3})\delta(v-\rho).
 \end{split}
 \end{equation}
 We decompose, using the Gevrey cutoff function $\Psi$ as in \eqref{GKA10}, for $v,\rho,w\in\R$,
 \begin{equation}\label{GKA27}
\Psi(\rho)k_\epsilon(v,\rho;w)=k^1_\epsilon(v,\rho;w)+k^2_\epsilon(v,\rho;w),
 \end{equation}
 where 
 \begin{equation}\label{GKA28}
 \begin{split}
 &\epsilon\partial_v^2k^1_\epsilon(v,\rho;w)+\epsilon\frac{\partial_wB^\ast(w)}{B^\ast(w)}\partial_vk^1_\epsilon(v,\rho;w)-i\frac{v}{(B^\ast(w))^2}k^1_\epsilon(v,\rho;w)\\
 &=\bigg\{\epsilon\bigg[\frac{\partial_vB^\ast(w)}{B^\ast(w)}-\frac{\partial_vB^\ast(v+w)}{B^\ast(v+w)}\bigg]\partial_v+\bigg[\frac{iv}{(B^\ast(w))^2}-\frac{iv}{(B^\ast(v+w))^2}\bigg]\bigg\}\Psi(\rho)k_\epsilon(v,\rho;w),
 \end{split}
 \end{equation}
 and
 \begin{equation}\label{GKA29}
 \begin{split}
 &\bigg[\epsilon\partial_v^2+\epsilon\frac{\partial_wB^\ast(w)}{B^\ast(w)}\partial_v-i\frac{v}{(B^\ast(w))^2}\bigg]k^2_\epsilon(v,\rho;w)=(\rho+i\epsilon^{1/3})\Psi(\rho)\delta(v-\rho).
 \end{split}
 \end{equation}
 Recalling \eqref{GKA12}, we have the bounds for $v,\rho,\xi\in\R$, 
  \begin{equation}\label{GKA30}
 \begin{split}
 &e^{\delta_0\langle\xi\rangle^{1/2}}\big|\widehat{\,\,k^1_{\epsilon,w_0}}(v,\rho;\xi)\big|\\
 &\lesssim \int_\R \epsilon^{-1/3}\langle \epsilon^{-1/3}v'\rangle^{1/2}|v'|\frac{\epsilon^{1/3}\langle\epsilon^{-1/3}v'\rangle}{|v'+i\epsilon^{1/3}|}e^{-\delta_0 \epsilon^{-1/2}||v'|^{1/2}v'-|\rho|^{1/2}\rho|} e^{-2\delta_0\epsilon^{-1/2}||v|^{1/2}v-|v'|^{1/2}v'|}\,dv'\\
 &+ \int_\R \epsilon^{-1/3}\langle \epsilon^{-1/3}v'\rangle^{1/2}\frac{\epsilon^{-1/3}|v'|^2\langle\epsilon^{-1/3}v'\rangle^{1/2}}{|v'+i\epsilon^{1/3}|} e^{-\delta_0 \epsilon^{-1/2}||v|^{1/2}v-|\rho|^{1/2}\rho|}  e^{-2\delta_0\epsilon^{-1/2}||v|^{1/2}v-|v'|^{1/2}v'|}\,dv'\\
 &\lesssim \langle\epsilon^{-1/3}v\rangle^{3/2}e^{-\delta_0 \epsilon^{-1/2}||v|^{1/2}v-|\rho|^{1/2}\rho|}.
 \end{split}
 \end{equation}
 Similarly, we have 
 \begin{equation}\label{GKA31}
 \big|\partial_v\widehat{\,\,k^1_{\epsilon,w_0}}(v,\rho;\xi)\big|\lesssim \epsilon^{-1/3}\langle\epsilon^{-1/3}v\rangle^2 e^{-\delta_0\langle \xi\rangle^{1/2}}e^{-\delta_0 \epsilon^{-1/2}||v|^{1/2}v-|\rho|^{1/2}\rho|}.
 \end{equation}
 We turn to the more singular component $k^2_{\epsilon}(v,\rho;w)$. Denote for $\rho\in\R$, 
 \begin{equation}\label{GKA31.1}
 \Psi_1(\rho):=\rho\Psi(\rho),
 \end{equation}
 then for $\gamma\in\R$,
$ \widehat{\,\,\Psi_1}(\gamma)=i\partial_\gamma\widehat{\,\Psi\,}(\gamma).$
 Taking Fourier transform of \eqref{GKA29} in both $v,\rho\in\R$, we obtain for $\alpha,\beta,w\in\R$ that
 \begin{equation}\label{GKA32}
 \begin{split}
 &-\epsilon\alpha^2\widetilde{\,\, k_\epsilon^2}(\alpha,\beta;w)+i\epsilon\alpha\frac{\partial_wB^\ast(w)}{B^\ast(w)}\widetilde{\,\, k_\epsilon^2}(\alpha,\beta;w)+\frac{1}{(B^\ast(w))^2}\partial_\alpha \widetilde{\,\, k_\epsilon^2}(\alpha,\beta;w)\\
 &=\widehat{\,\,\Psi_1}(\alpha+\beta)+i\epsilon^{1/3}\widehat{\,\,\Psi\,\,}(\alpha+\beta).
 \end{split}
 \end{equation}
 Define for $\gamma,\alpha,w\in\R$,
 \begin{equation}\label{GKA32.1}
 E_{xp}(\gamma,\alpha,w):=e^{-\epsilon/3(B^\ast(w))^2[\gamma^3-\alpha^3]+i(\epsilon/2)B^\ast\partial_wB^\ast(w)[\gamma^2-\alpha^2]}.
 \end{equation}
We obtain from \eqref{GKA32} that
\begin{equation}\label{GKA33}
\begin{split}
 &\widetilde{\,\,k_\epsilon^2}(\alpha,\beta;w)=-\int_{\alpha}^{\infty} (B^\ast(w))^2\Big[\widehat{\,\,\Psi_1}(\gamma+\beta)+i\epsilon^{1/3}\widehat{\,\Psi\,}(\gamma+\beta)\Big]E_{xp}(\gamma,\alpha,w)\,d\gamma\\
 &=i(B^\ast(w))^2\widehat{\,\Psi\,}(\alpha+\beta)+\int_{\alpha}^{\infty} (B^\ast(w))^2\widehat{\,\Psi\,}(\gamma+\beta)i\partial_\gamma E_{xp}(\gamma,\alpha,w)\,d\gamma\\
 &\quad-\int_\alpha^\infty i\epsilon^{1/3} (B^\ast(w))^2\widehat{\,\Psi\,}(\gamma+\beta) E_{xp}(\gamma,\alpha,w)\,d\gamma\\
&=-\int_{\alpha}^{\infty} i\Big[\epsilon\gamma^2(B^\ast(w))^2-i\epsilon \gamma B^\ast\partial_wB^\ast(w)\Big](B^\ast(w))^2\widehat{\,\Psi\,}(\gamma+\beta) E_{xp}(\gamma,\alpha,w)\,d\gamma\\
&\quad+i(B^\ast(w))^2\widehat{\,\Psi\,}(\alpha+\beta)-\int_{\alpha}^{\infty} i\epsilon^{1/3} (B^\ast(w))^2\widehat{\,\Psi\,}(\gamma+\beta) E_{xp}(\gamma,\alpha,w)\,d\gamma.
\end{split}
\end{equation}
Setting for $\alpha,\gamma,w\in\R$,
\begin{equation}\label{GKA34}
h(\gamma,\alpha;w):=E_{xp}(\gamma,\alpha,w){\bf 1}_{[0,\infty)}(\gamma-\alpha),
\end{equation} 
for $h^\ast\in\{\Psi(w-w_0)(B^\ast(w))^2h(\gamma,\alpha;w), \Psi(w-w_0)B^\ast\partial_wB^\ast(w)h(\gamma,\alpha;w)\}$, we have for a suitable $c_1\in(0,1)$,
\begin{equation}\label{GKA35}
\Big|\int_\R h^\ast(\gamma,\alpha;w)e^{-iw\eta}\,dw\Big|\lesssim e^{-c_1\epsilon(\gamma^3-\alpha^3)}{\bf 1}_{[0,\infty)}(\gamma-\alpha)e^{-\delta_0\langle \eta\rangle^{1/2}}.
\end{equation}
We obtain that 
\begin{equation}\label{GKA36}
\sup_{\eta\in\R}\Big[\big| \widetilde{\,\,k_\epsilon^2}(\alpha,\beta;\eta)\big|e^{\delta_0\langle\eta\rangle^{1/2}}\Big]\lesssim \int_{\alpha}^{\infty} \big[\epsilon(\gamma^2+|\gamma|)+\epsilon^{1/3}\big] e^{-c_1\epsilon(\gamma^3-\alpha^3)}|\widehat{\,\Psi\,}(\gamma+\beta)|\,d\gamma,
\end{equation}
and therefore
\begin{equation}\label{GKA37}
\sup_{\alpha\in\R}\int_\R\sup_{\eta\in\R}\Big[\big| \widetilde{\,\,k_\epsilon^2}(\alpha,\beta;\eta)\big|e^{\delta_0\langle\eta\rangle^{1/2}}\Big] d\beta\lesssim\sup_{\alpha\in\R}\int_{\alpha}^{\infty} \big[\epsilon\gamma^2+\epsilon^{1/3}\big]e^{-c_1\epsilon(\gamma^3-\alpha^3)}d\gamma\lesssim1.
\end{equation}
 
The proof of Proposition \ref{GKA11} is now complete.

 %
 %

\section{Limiting absorption principle for the Orr-Sommerfeld equation}\label{sec:lap}
In this section we prove the limiting absorption principle. We begin with the following bounds for solutions to the equation
\begin{equation}\label{LAP1}
\epsilon \partial_y^2w-\alpha w+i(b(y_0)-b(y))w=f(y),\quad y\in\R.
\end{equation}
In the above, we assume that $0<|\epsilon|<1/8$ and $\alpha\in\R$ with $\epsilon\alpha\ge0$.
\begin{proposition}\label{LAPM1}
Let $0<|\epsilon|<1/8$, $\alpha\in\R$ with $\epsilon\alpha\ge0$, $k\in\Z\backslash\{0\}$, $y_0\in\R$, and $f\in H^1(\R)$. Suppose that $w_{\epsilon,\alpha}(\cdot,y_0), \psi_{\epsilon,\alpha}(\cdot,y_0)\in H^2(\R)$ are the solution to the equations for $y\in\R$, 
\begin{equation}\label{LAPM2}
\begin{split}
&\epsilon \partial_y^2w_{\epsilon,\alpha}(y,y_0)-\alpha w_{\epsilon,\alpha}(y,y_0)+i(b(y_0)-b(y))w_{\epsilon,\alpha}(y,y_0)=f(y),\\
&-(k^2-\partial_y^2)\psi_{\epsilon,\alpha}(y,y_0)=w_{\epsilon,\alpha}(y,y_0).
\end{split}
\end{equation}
Then we have the following conclusions.

(i) We have the energy estimates
\begin{equation}\label{LAPM2.1}
\|(y-y_0)w_{\epsilon,\alpha}(y,y_0)\|_{L^2(y\in\R)}+\epsilon^{1/3}\|w_{\epsilon,\alpha}(y,y_0)\|_{L^2(y\in\R)}+\epsilon^{2/3}\|\partial_yw_{\epsilon,\alpha}(y,y_0)\|_{L^2(y\in\R)}\lesssim \|f\|_{L^2(\R)};
\end{equation}

(ii) Denoting $m:=\|h\|_{H^1(\R)}$, we have the pointwise bounds for $y\in\R$,
\begin{equation}\label{LAPM3}
\begin{split}
|w_{\epsilon,\alpha}(y,y_0)|&\lesssim |\epsilon|^{-1/3}\langle \epsilon^{-1/3}(y-y_0),\epsilon^{-1/3}\alpha\rangle^{-1}m,\\
|\partial_yw_{\epsilon,\alpha}(y,y_0)|&\lesssim \Big[|\epsilon|^{-1/2}\langle \epsilon^{-1/3}(y-y_0),\epsilon^{-1/3}\alpha\rangle^{-3/4}+|\epsilon|^{-2/3}\langle \epsilon^{-1/3}(y-y_0),\epsilon^{-1/3}\alpha\rangle^{-2}\Big]m,\\
|\partial^2_yw_{\epsilon,\alpha}(y,y_0)|&\lesssim \Big[|\epsilon|^{-5/6}\langle \epsilon^{-1/3}(y-y_0),\epsilon^{-1/3}\alpha\rangle^{-1/4}+|\epsilon|^{-1}\langle \epsilon^{-1/3}(y-y_0),\epsilon^{-1/3}\alpha\rangle^{-3/2}\Big]m;
\end{split}
\end{equation}

(iii) We have the bounds for the ``stream function" $\psi_{\epsilon,\alpha}$,
\begin{equation}\label{LAPM4}
\sup_{\xi\in\R}\Big[\langle k,\xi\rangle^2\big|\widehat{\,\,\psi_{\epsilon,\alpha}}(\xi,y_0)\big|\Big]\lesssim\|f\|_{H^1(\R)};
\end{equation}

(iv) In addition, we have the limiting behaviors

\begin{itemize}

\item If $\alpha_0\neq0$,
\begin{equation}\label{LAPM4.91}
\lim_{\epsilon\to0, \,z\to y_0,\, \alpha\to \alpha_0,\, \epsilon\alpha>0} w_{\epsilon,\alpha}(y,z)=\frac{-f(y)}{\alpha_0+i(b(y)-b(y_0))};
\end{equation}

\item 
\begin{equation}\label{LAPM4.1}
\lim_{\epsilon\to0+,\, \alpha\to 0+,\, z\to y_0} w_{\epsilon,\alpha}(y,z)={\rm P.V.}\frac{-f(y)}{i(b(y)-b(y_0))}-\pi \frac{f(y_0)}{b'(y_0)} \delta(y-y_0),
\end{equation}
for $y\in\R$, in the sense of distributions. 

\item Similarly for $y\in\R$, in the sense of distributions for $y\in\R$,
\begin{equation}\label{LAPM4.101}
\lim_{\epsilon\to0-,\, \alpha\to0-, \,z\to y_0} w_{\epsilon,\alpha}(y,z)={\rm P.V.}\frac{-f(y)}{i(b(y)-b(y_0))}+\pi \frac{f(y_0)}{b'(y_0)} \delta(y-y_0).
\end{equation}
\end{itemize}

\end{proposition}

\begin{remark}
We note that the bound 
\begin{equation}\label{LAPM3.1}
|\epsilon||\partial^2_yw_{\epsilon,\alpha}(y,y_0)|\lesssim \Big[|\epsilon|^{1/6}\langle \epsilon^{-1/3}(y-y_0),\epsilon^{-1/3}\alpha\rangle^{-1/4}+\langle \epsilon^{-1/3}(y-y_0),\epsilon^{-1/3}\alpha\rangle^{-3/2}\Big]\|f\|_{H^1(\R)}
\end{equation}
implies that the singular perturbation term $\epsilon \partial_y^2w$ is ``negligible" outside the ``critical region" $|y-y_0|\lesssim |\epsilon|^{1/3}$, for sufficiently small $\epsilon$.

\end{remark}

\begin{proof}
For concreteness, we assume that $0<\epsilon<1/8$ and correspondingly $\alpha\ge0$, as the other case is completely analogous. We organize the proof into several steps. 

{\it Step 1: Proof of \eqref{LAPM2.1}.} We first give the proof of \eqref{LAPM2.1}. By integration by parts, we have
\begin{equation}\label{LAPM0.34}
\begin{split}
&-\epsilon\int_\R|\partial_yw_{\epsilon,\alpha}(y,y_0)|^2dy-\alpha\int_\R |w_{\epsilon,\alpha}(y,y_0)|^2\,dy+i\int_\R (b(y_0)-b(y))|w_{\epsilon,\alpha}(y,y_0)|^2dy\\
&=\int_\R f(y)\overline{w_{\epsilon,\alpha}(y,y_0)}\,dy,\\
&-\epsilon \int_\R (b(y_0)-b(y))|\partial_yw_{\epsilon,\alpha}(y,y_0)|^2dy-\alpha\int_\R(b(y_0)-b(y))|w_{\epsilon,\alpha}(y,y_0)|^2dy\\
&+\epsilon\int_\R \partial_yw_{\epsilon,\alpha}(y,y_0)b'(y)\overline{w_{\epsilon,\alpha}(y,y_0)}\,dy+i\int_\R|b(y)-b(y_0)|^2|w_{\epsilon,\alpha}(y,y_0)|^2dy\\
&=\int_\R f(y)(b(y_0)-b(y))\overline{w_{\epsilon,\alpha}(y,y_0)}\,dy.
\end{split}
\end{equation}
It follows from \eqref{LAPM0.34} that
\begin{equation}\label{LAP35}
\begin{split}
&\epsilon\int_\R |\partial_yw_{\epsilon,\alpha}(y,y_0)|^2dy\leq \|f\|_{L^2(\R)}\|w_{\epsilon,\alpha}(y,y_0)\|_{L^2(y\in\R)},\\
&\int_\R |y-y_0|^2|w_{\epsilon,\alpha}(y,y_0)|^2dy\\
&\lesssim \epsilon\int_\R |\partial_yw_{\epsilon,\alpha}(y,y_0)||w_{\epsilon,\alpha}(y,y_0)|\,dy+\|f\|_{L^2(\R)}\||y-y_0|w_{\epsilon,\alpha}(y,y_0)\|_{L^2(y\in\R)}.
\end{split}
\end{equation}
By the interpolation inequality 
\begin{equation}\label{LAPM0.36}
\|w_{\epsilon,\alpha}(y,y_0)\|^2_{L^2(\R)}\lesssim \||y-y_0|w_{\epsilon,\alpha}(y,y_0)\|_{L^2(y\in\R)}\|\partial_yw_{\epsilon,\alpha}(y,y_0)\|_{L^2(y\in\R)},
\end{equation}
we obtain from \eqref{LAP35} that
\begin{equation}\label{LAPM0.365}
\begin{split}
\|w_{\epsilon,\alpha}(y,y_0)\|_{L^2(y\in\R)}&\lesssim  \||y-y_0|w_{\epsilon,\alpha}(y,y_0)\|_{L^2(y\in\R)}^{1/2}\|\partial_yw_{\epsilon,\alpha}(y,y_0)\|_{L^2(y\in\R)}^{1/2}\\
&\lesssim \epsilon^{-1/4}\|f\|_{L^2(\R)}^{1/4}       
\||y-y_0|w_{\epsilon,\alpha}(y,y_0)\|_{L^2(y\in\R)}^{1/2} \|w_{\epsilon,\alpha}(y,y_0)\|_{L^2(y\in\R)}^{1/4},
\end{split}
\end{equation}
which implies that 
\begin{equation}\label{LAPM0.366}
\|w_{\epsilon,\alpha}(y,y_0)\|_{L^2(y\in\R)}\lesssim \epsilon^{-1/3}\||y-y_0|w_{\epsilon,\alpha}(y,y_0)\|_{L^2(y\in\R)}^{2/3}\|f\|_{L^2(\R)}^{1/3}.
\end{equation}
It follows from \eqref{LAPM0.366} and \eqref{LAP35} that
\begin{equation}\label{LAPM0.37}
\begin{split}
&\int_\R |y-y_0|^2|w_{\epsilon,\alpha}(y,y_0)|^2dy+\epsilon^{4/3}\int_\R |\partial_yw_{\epsilon,\alpha}(y,y_0)|^2dy\\
&\lesssim_\sigma  \|f\|_{L^2(\R)}\||y-y_0|w_{\epsilon,\alpha}(y,y_0)\|_{L^2(y\in\R)}+\epsilon^{1/3}\|f\|_{L^2(\R)}\|w_{\epsilon,\alpha}(y,y_0)\|_{L^2(y\in\R)}\\
&\lesssim_\sigma \|f\|_{L^2(\R)}\||y-y_0|w_{\epsilon,\alpha}(y,y_0)\|_{L^2(y\in\R)}\\
&\qquad+\epsilon^{1/3}\|f\|_{L^2(\R)}\||y-y_0|w_{\epsilon,\alpha}(y,y_0)\|^{1/2}_{L^2(y\in\R)}\|\partial_yw_{\epsilon,\alpha}(y,y_0)\|^{1/2}_{L^2(y\in\R)}\\
&\lesssim_\sigma \|f\|_{L^2(\R)}\bigg[\int_\R |y-y_0|^2|w_{\epsilon,\alpha}(y,y_0)|^2dy+\epsilon^{4/3}\int_\R|\partial_yw_{\epsilon,\alpha}(y,y_0)|^2dy\bigg]^{1/2}.
\end{split}
\end{equation}
The desired bounds \eqref{LAPM2.1} follow from \eqref{LAPM0.37} and \eqref{LAPM0.36}.

 We now turn to the proof of (ii)-(iv). We can assume that $f\in H^1(\R)$ with $\|f\|_{H^1(\R)}=1$.

{\it Step 2: Proof of \eqref{LAPM3}.} The bounds on $w_{\epsilon,\alpha}$ in \eqref{LAPM3} follow from the kernel estimates \eqref{GKA2} and $\|f\|_{L^\infty(\R)}\lesssim1$. To prove the bounds on the derivatives of $w_{\epsilon,\alpha}$ in \eqref{LAPM3}, we take derivatives in $y$ in \eqref{LAP1} to obtain that  for $y\in\R$,
\begin{equation}\label{LAP7}
\epsilon \partial_y^2\partial_yw_{\epsilon,\alpha}(y,y_0)-\alpha\partial_yw_{\epsilon,\alpha}(y,y_0)+i(b(y_0)-b(y))\partial_yw_{\epsilon,\alpha}(y,y_0)=\partial_yf+ib'(y)w_{\epsilon,\alpha}(y,y_0).
\end{equation}
The desired bounds then follow from the kernel estimates \eqref{GKA2}, $\|\partial_yf\|_{L^2(\R)}\leq1$ and the pointwise bounds on $w_{\epsilon,\alpha}$ in \eqref{LAPM3}.

{\it Step 3: Proof of \eqref{LAPM4}.} To capture the singular behavior of $w_{\epsilon,\alpha}$ for $|y-y_0|\lesssim\epsilon^{1/3}$ more precisely, we rewrite \eqref{LAPM2} as 
\begin{equation}\label{LAP12}
\begin{split}
&\epsilon\partial_y^2w_{\epsilon,\alpha}(y,y_0)-\alpha w_{\epsilon,\alpha}(y,y_0)-ib'(y_0)(y-y_0)w_{\epsilon,\alpha}(y,y_0)\\
&=f(y)-i\big[b'(y_0)(y-y_0)-(b(y)-b(y_0))\big]w_{\epsilon,\alpha}(y,y_0),
\end{split}
\end{equation}
and decompose for $y\in\R$,
\begin{equation}\label{LAP13}
w_{\epsilon,\alpha}(y,y_0):=w_{1\epsilon,\alpha}(y,y_0)+w_{2\epsilon,\alpha}(y,y_0),
\end{equation}
where $w_{1\epsilon,\alpha}$ and $w_{2\epsilon,\alpha}$ solve the equations
\begin{equation}\label{LAP14}
\begin{split}
&\epsilon\partial_y^2w_{1\epsilon,\alpha}(y,y_0)-\alpha w_{1\epsilon,\alpha}(y,y_0)-ib'(y_0)(y-y_0)w_{1\epsilon,\alpha}(y,y_0)=f(y_0),\\
&\epsilon\partial_y^2w_{2\epsilon,\alpha}(y,y_0)-\alpha w_{2\epsilon,\alpha}(y,y_0)-ib'(y_0)(y-y_0)w_{2\epsilon,\alpha}(y,y_0)\\
&=f(y)-f(y_0)-i\big[b'(y_0)(y-y_0)-(b(y)-b(y_0))\big]w_{\epsilon,\alpha}(y,y_0).\\
\end{split}
\end{equation}
It follows from the kernel estimates \eqref{GKA2}, $\|f\|_{H^1(\R)}=1$, and the bound \eqref{LAPM3} that for $y\in\R$ with $|y-y_0|<1$,
\begin{equation}\label{LAP15}
|w_{2\epsilon,\alpha}(y,y_0)|\lesssim \epsilon^{-1/6}\langle\epsilon^{-1/3}(y-y_0),\epsilon^{-1/3}\alpha\rangle^{-1/2}.
\end{equation}
By rescaling, we have for $y\in\R$,
\begin{equation}\label{LAP16}
w_{1\epsilon,\alpha}(y)=\epsilon^{-1/3}f(y_0)[b'(y_0)]^{-2/3}W\big(\epsilon^{-1/3}(b'(y_0))^{1/3}(y-y_0)\big),
\end{equation}
where $W$ satisfies the equation for $Y\in\R$,
\begin{equation}\label{LAP17}
\partial_Y^2W(Y)-\epsilon^{-1/3}\alpha(b'(y_0))^{-2/3}W(Y)-iYW(Y)=1.
\end{equation}
In the above we have suppressed the dependence of $W$ on $\epsilon, \alpha$. Equation \eqref{LAP17} can be solved explicitly using Fourier transforms, and we have
\begin{equation}\label{LAP18}
\widehat{W}(\xi)=-\sqrt{2\pi}\,e^{\xi^3/3+\epsilon^{-1/3}\alpha(b'(y_0))^{-2/3}\xi}\,{\bf 1}_{(-\infty,0]}(\xi),\qquad{\rm for}\,\,\xi\in\R.
\end{equation}
The bounds \eqref{LAPM2.1}-\eqref{LAPM3}, the decomposition \eqref{LAP13} and the bounds \eqref{LAP15}, the identity \eqref{LAP16} and the bound \eqref{LAP18} imply that
\begin{equation}\label{LAP19}
\sup_{\xi\in\R}|\widehat{\,\,w_{\epsilon,\alpha}}(\xi)|\lesssim_\sigma 1.
\end{equation}
The desired bound \eqref{LAPM4} follows from \eqref{LAP19}.

{\it Step 4: Proof of \eqref{LAPM4.91}-\eqref{LAPM4.101}.} We first note that \eqref{LAPM4.91} follows from the bound \eqref{LAPM3} and the equation \eqref{LAPM2}. 

We now turn to the proof of the limit \eqref{LAPM4.1}. By equation \eqref{LAPM2} and the bounds \eqref{LAPM3}, we only need to consider the region $|y-z|<1/2$. We fix an even cutoff function $\Phi\in C_0^{\infty}(-2,2)$ with $\Phi\equiv 1$ on $[-1,1]$. For any $\varphi\in C_0^\infty(\R)$, we calculate, using the bound \eqref{LAPM3}, equation \eqref{LAPM2} and \eqref{LAP13}-\eqref{LAP18} that for any $\varphi\in C_0^\infty(-1,1)$,
\begin{equation}\label{LAP20}
\begin{split}
&\lim_{\epsilon,\alpha\to0+,z\to y_0}\int_\R w_{\epsilon,\alpha}(y,z)\varphi(y)\,dy=\lim_{\epsilon,\alpha\to0+,z\to y_0}\bigg[\int_\R \frac{-f(y)\varphi(y)}{i(b(y)-b(z))}(1-\Phi((y-z)/(\epsilon+\alpha)^{1/4}))\,dy\\
&\hspace{3.1in}+\varphi(y_0)\int_\R w_{1\epsilon,\alpha}(y,z)\Phi((y-z)/(\epsilon+\alpha)^{1/4})\,dy\bigg],
\end{split}
\end{equation}
and therefore,
\begin{equation}\label{LAP20.01}
\begin{split}
&\lim_{\epsilon,\alpha\to0+,z\to y_0}\int_\R w_{\epsilon,\alpha}(y,z)\varphi(y)\,dy-{\rm P.V.}\int_\R \frac{-f(y)\varphi(y)}{i(b(y)-b(y_0))}\,dy\\
&=\varphi(y_0)\lim_{\epsilon,\alpha\to0+,z\to y_0}\int_\R w_{1\epsilon,\alpha}(y,z)\Phi((y-z)/(\epsilon+\alpha)^{1/4})\,dy\\
&=\varphi(y_0)f(y_0)[b'(y_0)]^{-2/3}\lim_{\epsilon,\alpha\to0+,z\to y_0}\int_\R \epsilon^{-1/3}W\big(\epsilon^{-1/3}(b'(z))^{1/3}(y-z)\big)\Phi\big((y-z)/(\epsilon+\alpha)^{1/4}\big)\,dy\\
&=-\sqrt{2\pi}\,\varphi(y_0)\frac{f(y_0)}{b'(y_0)}\int_{-\infty}^0\widehat{\,\Phi\,}(\xi)\,d\xi=-\pi\varphi(y_0)\frac{f(y_0)}{b'(y_0)}.
\end{split}
\end{equation}
We note that the cutoff length we used, $\epsilon^{1/4}$, is somewhat arbitrary, and we can use any cutoff scale $d$ with $\max\{\epsilon^{1/3},\alpha\}\ll d\ll 1$, to separate the critical region which is of the size $\epsilon^{1/3}$ from the more regular region. Since $\varphi\in C_0^\infty(-1,1)$ is arbitrary, we obtain the desired identity \eqref{LAP20}.
\end{proof}

For $\epsilon\in(-1/8,1/8)\backslash\{0\}$, $\alpha\in\R$ with $\epsilon\alpha\ge0$ and $y_0\in\R$, we define the operator $T_{\epsilon,\alpha,y_0}: H^1(\R)\to H^1(\R)$ as follows. For any $f\in H^1(\R)$, suppose $w_{\epsilon,\alpha}, \psi_{\epsilon,\alpha}$ are given as in \eqref{LAPM2}. Then we define
\begin{equation}\label{LAP21}
T_{\epsilon,\alpha,y_0}f(y):=\psi_{\epsilon,\alpha}(y,y_0),\qquad{\rm for}\,\,y\in\R.
\end{equation}
The bounds \eqref{LAPM3}-\eqref{LAPM4} imply that $T_{\epsilon,\alpha,y_0}: H^1(\R)\to H^1(\R)$ is compact. 

We are now ready to establish the following limiting absorption principle which plays a fundamental role in controlling the solutions to the Orr-Sommerfeld equation in the high Reynolds number regime. 
\begin{proposition}\label{LAP22}
There exist $\epsilon_0>0$ and $\kappa>0$, such that for all $\epsilon\in(-\epsilon_0,\epsilon_0)\backslash\{0\}$, $\alpha\in\R$ with $\epsilon\alpha\ge0$, $k\in\Z\backslash\{0\}$, $y_0\in\R$, and $\psi\in H^1(\R)$, we have the bound
\begin{equation}\label{LAP23}
\|\psi+T_{\epsilon,\alpha,y_0}(ib''\psi)\|_{H_k^1(\R)}\ge \kappa \|\psi\|_{H_k^1(\R)}.
\end{equation}
\end{proposition}

\begin{proof}
For the sake of concreteness, we assume that $\epsilon>0$, as the other case is completely analogous. By the bound \eqref{LAPM4} we can assume that $k=k_0\in\Z\backslash\{0\}$. For $\psi\in H^1(\R)$, assume that $w_{\epsilon,\alpha}(\cdot,y_0)\in H^2(\R)$ is the solution to 
$$\epsilon\partial_y^2w_{\epsilon,\alpha}(y,y_0)+i(b(y_0)-b(y))w_{\epsilon,\alpha}(y,y_0)=-ib''(y)\psi(y),\quad{\rm for}\,\,y\in\R.$$
The kernel estimates \eqref{GKA2} and the support assumption ${\rm supp}\,b''\Subset[-1/\sigma_0,1/\sigma_0]$ imply that for some $c_0=c_0(\sigma)\in(0,1)$ and all $|y_0|>2/\sigma_0$,
\begin{equation}\label{LAP23.01}
|w_{\epsilon,\alpha}(y,y_0)|\lesssim\epsilon^{-1/2}\langle y_0\rangle^{-1/2}\int_\R e^{-c_0\epsilon^{-1/2}|y_0|^{1/2}|y-z|}|\psi(z)|\,dz,\quad y\in\R.
\end{equation}
It follows from \eqref{LAP23.01} that 
\begin{equation}\label{LAP23.02}
\|w_{\epsilon,\alpha}\|_{L^2(\R)}\lesssim \langle y_0\rangle^{-1}\|\psi\|_{H^1_k(\R)}.
\end{equation}
The desired bounds \eqref{LAP23} then follow if $|y_0|\gg1/\sigma_0$. Thus we can assume that $|y_0|\lesssim1$. Similarly, by the bound \eqref{LAPM3} we can assume that $0<\alpha\lesssim1$. In summary, it suffices to consider the case
\begin{equation}\label{LAP23.001}
k=k_0\in\Z\backslash\{0\},\quad |y_0|\lesssim1, \quad 0<\alpha\lesssim1.
\end{equation}
 
Suppose the inequality \eqref{LAP23} does not hold, then we can find a sequence $\epsilon_j\to0+$, $0<\alpha_j\lesssim1$, $y_{0j}\in\R$ with $|y_{0j}|\lesssim1$, and $\psi_j\in H^1(\R)$ with $\|\psi_j\|_{H^1(\R)}=1$ for $j\ge1$, such that 
\begin{equation}\label{LAP24}
\|\psi_{j}+T_{\epsilon_j,\alpha_j,y_{0j}}(ib''\psi_j)\|_{H^1(\R)}\to 0+, \qquad{\rm as}\,\,j\to\infty.
\end{equation}
By \eqref{LAP23.001} and the bound \eqref{LAPM4}, we can assume (by passing to a subsequence) that for some $\psi\in H^1(\R)$ with $\|\psi\|_{H^1(\R)}=1$,
\begin{equation}\label{LAP25}
y_{0j}\to y_0,\quad \alpha_j\to\alpha\ge0, \quad {\rm and}\quad \lim_{j\to\infty}\|\psi_{j}- \psi\|_{H^1(\R)}=0.
\end{equation}
We conclude from \eqref{LAP24}-\eqref{LAP25} that
\begin{equation}\label{LAP26}
\psi(y)+i\lim_{j\to\infty}T_{\epsilon_j,\alpha_j,y_{0j}}(b''\psi)(y)=0, \qquad{\rm for}\,\,y\in\R.
\end{equation}

We distinguish two cases, $\alpha=0$ and $\alpha>0$, and focus first on the case $\alpha=0$.
Applying the differential operator $-k^2+\partial_y^2$ to \eqref{LAP26}, and using the identity \eqref{LAPM4.1}, we get that in the sense of distributions,
\begin{equation}\label{LAP27}
(-k^2+\partial_y^2)\psi(y)-{\rm P.V.}\frac{b''(y)\psi(y)}{b(y)-b(y_0)}-i\pi \frac{b''(y_0)}{b'(y_0)}\psi(y_0)\delta(y-y_0)=0.
\end{equation}
Multiplying $\overline{\psi}$, integrating over $\R$ and taking the real part, we conclude that
\begin{equation}\label{LAP28}
b''(y_0)\psi(y_0)=0.
\end{equation}
Hence 
\begin{equation}\label{LAP29}
w:=(-k^2+\partial_y^2)\psi\in L^2(\R),
\end{equation}
and
\begin{equation}\label{LAP30}
(b(y)-b(y_0))w-b''(y)\psi=0.
\end{equation}
\eqref{LAP30} implies that $w\in L^2(\R)$ is an eigenfunction for the eigenvalue $b(y_0)$ of the operator $L_k$, which is a contradiction to our assumption that there are no discrete eigenvalues for $L_k$ in our main assumption \ref{MaAs}.

The case $\alpha>0$ is similar, and we can obtain a contradiction, using the identity \eqref{LAPM4.91} to get an eigenvalue $b(y_0)+i\alpha$ for $L_k$, and our assumption that there is no discrete eigenvalues for $L_k$. The theorem is now proved.

\end{proof}

\begin{remark}\label{LAP30.01}
Proposition \ref{LAP22} when $\alpha=0$ can be equivalently formulated as the following statement. The solution $\psi\in H^1(\R)$ to the equation 
\begin{equation}\label{LAP31}
\epsilon\partial_y^2(\partial_y^2-k^2)\psi-i(b(y)-b(y_0))(\partial_y^2-k^2)\psi+ib''(y)\psi=f(y),\quad{\rm for}\,\,y\in\R,
\end{equation}
can be bounded by the solution $\psi^\ast\in H^1(\R)$ to the equation
\begin{equation}\label{LAP32}
\epsilon\partial_y^2(\partial_y^2-k^2)\psi-i(b(y)-b(y_0))(\partial_y^2-k^2)\psi=f(y),\quad{\rm for}\,\,y\in\R,
\end{equation}
in the sense that 
\begin{equation}\label{LAP33}
\|\psi\|_{H^1_k(\R)}\lesssim \|\psi^\ast\|_{H^1_k(\R)}.
\end{equation}

\end{remark}


\section{Bounds on the spectral density function in Gevrey spaces}\label{sec:bts}
In this section we use the limiting absorption principle and commutator argument to establish bounds on the spectral density function in Gevrey spaces.

Recall the definitions \eqref{int11}-\eqref{int13}, and the equations \eqref{int12}. Assume that $k\in\Z\backslash\{0\}$ and $\nu\in(0,\nu_\ast)$ with $\nu_\ast$ being sufficiently small. We need to study the following equations for $\Upsilon_{k,\nu}(v,w):=\Omega_{k,\nu}(v+w,w)$ and $\Theta_{k,\nu}(v,w)$ with $v,w\in\R$, 
\begin{equation}\label{BSDG1}
\begin{split}
(\nu/k)(B^\ast(v+w))^2\partial_v^2\Upsilon_{k,\nu}(v,w)+(\nu/k) B^\ast(v+w)\partial_vB^\ast(v+w)&\partial_v\Upsilon_{k,\nu}(v,w)-iv\Upsilon_{k,\nu}(v,w)\\
\qquad+iB^\ast(v+w)\partial_vB^\ast(v+w)\Theta_{k,\nu}(v,w)&=F_{0k}(v+w),\\
\Big[-k^2+(B^\ast(v+w))^2\partial_v^2+B^\ast(v+w)\partial_vB^\ast(v+w)\partial_v\Big]\Theta_{k,\nu}(v,w)&=\Upsilon_{k,\nu}(v,w).
\end{split}
\end{equation}

The limiting absorption principle, see Proposition \ref{LAP22}, can be reformulated for the system of equations \eqref{int12} in the variables $v,w$ as follows. 
\begin{proposition}\label{BSDG2}
There exists $\nu_\ast\in(0,1)$ sufficiently small, such that the following statement holds. Assume that $\nu\in(0,\nu_\ast)$, $k\in\Z\backslash\{0\}$, $w\in\R$ and $f\in L^2(\R)$. Denote $\epsilon=\nu/k$. Let $h, \varphi\in H^2(\R)$ be the solution to the system of equations for $v\in\R$,
\begin{equation}\label{BSDG3}
\begin{split}
(\nu/k)(B^\ast(v+w))^2\partial_v^2h(v)+(\nu/k) B^\ast(v+w)\partial_vB^\ast(v+w)&\partial_vh(v)-iv\varphi(v)\\
\qquad+iB^\ast(v+w)\partial_vB^\ast(v+w)\varphi(v)&=f(v+w),\\
\Big[-k^2+(B^\ast(v+w))^2\partial_v^2+B^\ast(v+w)\partial_vB^\ast(v+w)\partial_v\Big]\varphi(v)&=h(v).
\end{split}
\end{equation}
Define 
\begin{equation}\label{BSDG4}
F(v):=\int_{\R^2}\mathcal{G}_k(v,v';w)k_\epsilon(v',\rho;w)\frac{1}{(B^\ast(\rho+w))^2}\frac{f(\rho+w)}{\rho+i\epsilon^{1/3}} d\rho dv'.
\end{equation}

(i) Then we have the bound
\begin{equation}\label{BSDG5}
\|\varphi\|_{H^1_k(\R)}\lesssim \|F\|_{H^1_k(\R)}.
\end{equation}

(ii) Equivalently, we can reformulate the equation \eqref{BSDG3} in the integral form as 
\begin{equation}\label{BSDGL1}
\varphi(v)+i\int_{\R^2}\mathcal{G}_{k}(v,v';w)k_{\epsilon}(v',\rho;w)\frac{\partial_\rho B^\ast(\rho+w)}{B^\ast(\rho+w)} \frac{\varphi(\rho)}{\rho+i\epsilon^{1/3}} d\rho dv'=F(v),\quad {\rm for}\,\,v\in\R.
\end{equation}
Then we have the bound \eqref{BSDG5} for the solution $\varphi$.

\end{proposition}

\begin{proof}
The bound \eqref{BSDG5} follows from Proposition \ref{LAP22} (see also remark \ref{LAP30.01}), and the change of variables \eqref{int10}.

\end{proof}

To use Proposition \ref{BSDG2} to obtain bounds on $\Theta_{k,\nu}$, we need the following lemmas. 

\begin{lemma}\label{BSDG6}
There exists $\nu_\ast\in(0,1)$ sufficiently small, such that the following statement holds. Assume that $\nu\in(0,\nu_\ast)$, $k\in\Z\backslash\{0\}$, $w\in\R$ and $f\in L^2(\R)$. Denote $\epsilon=\nu/k$. Define for $v,w\in\R$,
\begin{equation}\label{BSDG7}
F(v,w):=\int_{\R^2}\mathcal{G}_k(v,v';w)k_\epsilon(v',\rho;w)\frac{1}{(B^\ast(\rho+w))^2}\frac{f(\rho+w)}{\rho+i\epsilon^{1/3}} d\rho dv'.
\end{equation}
Then we have the bound
\begin{equation}\label{BSDG8}
\|\langle k,\partial_v\rangle F\|_{L^2(\R^2)}\lesssim \|f\|_{L^2(\R)}.
\end{equation}

\end{lemma}

\begin{proof}
We choose a Gevrey cutoff function $\Psi^\ast\in C_0^\infty((-6,6))$ such that 
\begin{equation}\label{BSDG13}
\Psi^\ast\equiv 1\,\,{\rm on }\,\,[-5,5], \quad \sup_{\xi\in\R}\Big[e^{\langle\xi\rangle^{3/4}}\widehat{\,\Psi^\ast}(\xi)\Big]\lesssim1.
\end{equation}
and also a Gevrey cutoff function $\Psi^{\ast\ast}\in C_0^\infty((-20,20))$ such that 
\begin{equation}\label{BSDG13}
\Psi^{\ast\ast}\equiv 1\,\,{\rm on }\,\,[-15,15], \quad \sup_{\xi\in\R}\Big[e^{\langle\xi\rangle^{3/4}}\widehat{\,\,\Psi^{\ast\ast}}(\xi)\Big]\lesssim1.
\end{equation}
We can decompose, using a Gevrey cutoff function $\Psi$ satisfying \eqref{GKA10} and the definitions for $k^1_\epsilon, k^2_\epsilon$ from Proposition \ref{GKA11}, for $v,w\in\R$,
\begin{equation}\label{BSDG9}
\begin{split}
F(v,w)=F_{N}(v,w)+F_{L}(v,w)+F_1(v,w)+F_2(v,w),
\end{split}
\end{equation}
where
\begin{equation}\label{BSDG10}
F_{N}(v,w):=\int_{\R^2}\mathcal{G}_k(v,v';w)(1-\Psi(\rho))k_\epsilon(v',\rho;w)\frac{1}{(B^\ast(\rho+w))^2}\frac{f(\rho+w)}{\rho+i\epsilon^{1/3}} d\rho dv',
\end{equation}
\begin{equation}\label{BSDG10.5}
F_{L}(v,w):=\int_{\R^2}\mathcal{G}_k(v,v';w)(1-\Psi^{\ast}(v'))\Psi(\rho)k_\epsilon(v',\rho;w)\frac{1}{(B^\ast(\rho+w))^2}\frac{f(\rho+w)}{\rho+i\epsilon^{1/3}} d\rho dv',
\end{equation}
\begin{equation}\label{BSDG11}
F_{1}(v,w):=\int_{\R^2}\mathcal{G}_k(v,v';w)k^1_\epsilon(v',\rho;w)\frac{\Psi^{\ast}(v')}{(B^\ast(\rho+w))^2}\frac{\Psi^\ast(\rho)f(\rho+w)}{\rho+i\epsilon^{1/3}} d\rho dv',
\end{equation}
and
\begin{equation}\label{BSDG12}
F_{2}(v,w):=\int_{\R^2}\mathcal{G}_k(v,v';w)k^2_\epsilon(v',\rho;w)\frac{\Psi^{\ast}(v')}{(B^\ast(\rho+w))^2}\frac{\Psi^\ast(\rho)f(\rho+w)}{\rho+i\epsilon^{1/3}} d\rho dv'.
\end{equation}
It suffices to prove that for $\ast\in\{N,L,1,2\}$,
\begin{equation}\label{BSDG14}
\|\langle k,\partial_v\rangle F_\ast\|_{L^2(\R^2)}\lesssim \|f\|_{L^2(\R)}.
\end{equation}
We first prove \eqref{BSDG14} for $F_{N}$. Using the bound \eqref{GKA8}, \eqref{BSDG4.006} and \eqref{BSDG4.007}, we have
\begin{equation}\label{BSDG15}
\begin{split}
|k||F_{N}(v,w)|+|\partial_vF_{N}(v,w)|\lesssim \int_\R e^{-\delta_0|v-\rho|}\langle\rho\rangle^{-1}|f(\rho+w)|\,d\rho,
\end{split}
\end{equation}
from which the desired bounds \eqref{BSDG14} for $F_{N}$ follows. 

To prove \eqref{BSDG14} for $F_L$, we first note the pointwise bounds for $v',\rho,w\in\R$ with $|v'-\rho|>1$,
\begin{equation}\label{BSDG15.001}
\left|k_\epsilon(v',\rho;w)\frac{1}{\rho+i\epsilon^{1/3}}\right|\lesssim \epsilon^{-1/3}\langle\epsilon^{-1/3}v',\epsilon^{-1/3}\rho\rangle^{1/2}e^{-\delta_0\langle\epsilon^{-1/3}v',\epsilon^{-1/3}\rho\rangle^{1/2}\epsilon^{-1/3}|v'-\rho|}.
\end{equation}
Then using the bound \eqref{GKA8} and \eqref{BSDG4.006}-\eqref{BSDG4.007} we have
\begin{equation}\label{BSDG15}
\begin{split}
|k||F_{L}(v,w)|+|\partial_vF_{L}(v,w)|\lesssim \int_\R e^{-\delta_0|v-\rho|}\langle\rho\rangle^{-1}|f(\rho+w)|\,d\rho,
\end{split}
\end{equation}
from which the desired bounds \eqref{BSDG14} for $F_{L}$ follows. 

To prove \eqref{BSDG14} for $F_1$, using the inequality \eqref{GKA14.2}, we can bound 
\begin{equation}\label{BSDG16}
\begin{split}
&|k||F_{1}(v,w)|+|\partial_vF_{1}(v,w)|\\
&\lesssim  \int_{\R^2} e^{-\delta_0|v-v'|}\epsilon^{-1/3}\langle \epsilon^{-1/3}\rho\rangle^{1/2}e^{-\delta_0\langle\epsilon^{-1/3}v',\epsilon^{-1/3}\rho\rangle^{1/2}\epsilon^{-1/3}|v'-\rho|}\Psi^\ast(\rho)|f(\rho+w)|\,dv' d\rho\\
&\lesssim \int_\R e^{-\delta_0|v-\rho|}\langle\rho\rangle^{-1}|f(\rho+w)|\,d\rho,
\end{split}
\end{equation}
from which the desired bounds \eqref{BSDG14} for $F_{1}$ follows. 

We now turn to the proof of \eqref{BSDG14} for $F_2$, which is the most difficult case due to the singular factor $(\rho+i\epsilon^{1/3})^{-1}$ in \eqref{BSDG12}. Define for $j\in\Z$, $\rho,v,w\in\R$,
\begin{equation}\label{BSDG16.1}
F_{2j}(v,w):=F_{2}(v,w)\Psi(w-j),\quad f_j(\rho):=f(\rho)\Psi^{\ast\ast}(\rho-j),\quad k^2_{\epsilon,j}(v',\rho;w):=k^2_{\epsilon}(v',\rho;w)\Psi^{\ast}(w-j).
\end{equation}
Taking Fourier transform in $v, w$, we can bound for $j\in\Z, \xi,\eta\in\R$,
\begin{equation}\label{BSDG17}
\begin{split}
&|\widehat{\,\,F_{2j}}(\xi,\eta)|\\
&\lesssim  \int_{\R^5}\big|\widetilde{\,\,\mathcal{G}_{k,j}}(\xi,\alpha_1;\alpha_2)\big|\big|\widetilde{\,\,k^2_{\epsilon,j}}(-\alpha_1,\eta-\alpha_2-\alpha_3+\gamma;\alpha_3)\big|e^{-\delta_0\langle\beta\rangle^{1/2}}\big|\widehat{\,f_j}(\eta-\beta-\alpha_2-\alpha_3)\big|\,d\Xi\\
&\lesssim \int_{\R^5}\frac{e^{-\delta_0\langle\xi+\alpha_1,\alpha_2,\alpha_3,\beta\rangle^{1/2}}}{k^2+\xi^2}\Big|e^{\delta_0\langle\alpha_3\rangle^{1/2}}\widetilde{\,\,k^2_{\epsilon,j}}(-\alpha_1,\eta-\alpha_2-\alpha_3+\gamma;\alpha_3)\Big|\big|\widehat{\,f_j}(\eta-\beta-\alpha_2-\alpha_3)\big|\,d\Xi.
\end{split}
\end{equation}
In the above, we have used the notation $d\Xi=d\alpha_1d\alpha_2d\alpha_3d\beta d\gamma$ and the bounds \eqref{BSDG4.005}. The desired bounds \eqref{BSDG14} for $F_2$ now follow from \eqref{BSDG17} and the bounds \eqref{GKA14.3}, together with the inequality
\begin{equation}\label{BSDG18}
\sum_{j\in\Z}\|f_j\|_{L^2(\R)}^2\lesssim \|f\|_{L^2(\R)}^2.
\end{equation}
\end{proof}

\begin{lemma}\label{BSDG19}
There exists $\nu_\ast\in(0,1)$ sufficiently small, such that the following statement holds. Assume that $\nu\in(0,\nu_\ast)$, $k\in\Z\backslash\{0\}$ and $w\in\R$. Denote $\epsilon=\nu/k$. Suppose that $f\in L^2(\R)$ satisfies for some $\mu\in(0,1)$ with $\mu\ll \delta_0$,
\begin{equation}\label{BSDG19.1}
\big\|e^{\mu\langle k,\xi\rangle^{1/2}}\widehat{f\,}(\xi)\big\|_{L^2(\R)}<\infty.
\end{equation}
 Define for $v,w\in\R$,
\begin{equation}\label{BSDG20}
F(v,w):=\int_{\R^2}\mathcal{G}_k(v,v';w)k_\epsilon(v',\rho;w)\frac{1}{(B^\ast(\rho+w))^2}\frac{f(\rho+w)}{\rho+i\epsilon^{1/3}} d\rho dv'.
\end{equation}
Then we have bound
\begin{equation}\label{BSDG21}
\big\|\langle k,\xi\rangle e^{\mu\langle k,\eta\rangle^{1/2}} \widetilde{\,F\,}(\xi,\eta)\big\|_{L^2(\R^2)}\lesssim \big\|e^{\mu\langle k,\xi\rangle^{1/2}}\widehat{\,f\,\,}(\xi)\big\|_{L^2(\R)}.
\end{equation}

\end{lemma}

\begin{proof}
We decompose $F$ as in \eqref{BSDG9}-\eqref{BSDG12}, and define for $j\in\Z$, $\ast\in\{N,L, 1,2\}$ and $v, w\in\R$,
\begin{equation}\label{BSDG22}
F_{\ast}^j(v,w):=F_\ast(v,w)\Psi(w-j), \quad f_j(\rho):=\frac{f(\rho)}{(B^\ast(\rho+w))^2}\Psi^{\ast\ast}(\rho-j).
\end{equation}
It suffices to prove the inequality for $\ast\in\{N,L,1,2\}$,
\begin{equation}\label{BSDG23}
\big\|\langle k,\xi\rangle e^{\mu\langle k,\eta\rangle^{1/2}} \widetilde{\,\,F_\ast}(\xi,\eta)\big\|_{L^2(\R^2)}\lesssim \big\|e^{\mu\langle k,\xi\rangle^{1/2}}\widehat{\,f\,}(\xi)\big\|_{L^2(\R)}.
\end{equation}

We first establish \eqref{BSDG23} for $\ast=N$. For $j\in\Z$, it follows from \eqref{BSDG4.005} and the bound \eqref{GKA13} that
\begin{equation}\label{BSDG24}
\begin{split}
&e^{\mu\langle k,\eta\rangle^{1/2}}\bigg[|k|\Big|\widehat{\,\,F_{N}^j}(v,\eta)\Big|+\Big|\partial_v\widehat{\,\,F_{N}^j}(v,\eta)\Big|\bigg]\\
&\lesssim \sum_{\ell\in\Z}\int_{\R^3}e^{-\delta_0|v-\rho|}\Psi(\rho-\ell)\langle\rho\rangle^{-1}e^{\mu\langle k,\eta-\alpha-\beta\rangle^{1/2}}\big|\widehat{\,\,f_{j+\ell}}(\eta-\alpha-\beta)\big|e^{-\delta_0\langle\alpha,\beta\rangle^{1/2}}\,d\alpha d\beta d\rho.
\end{split}
\end{equation}
The desired bound \eqref{BSDG23} for $F_{N}$ follows from \eqref{BSDG24}.

The bounds on $F_L$ follow similarly, using also the pointwise inequality \eqref{BSDG15.001}.

We now turn to the proof of \eqref{BSDG23} for $F_1$. Using \eqref{BSDG4.005} and the bound \eqref{GKA14.2} we obtain that for $j\in\Z$,
\begin{equation}\label{BSDG25}
\begin{split}
&e^{\mu\langle k,\eta\rangle^{1/2}}\bigg[|k|\Big|\widehat{\,F_{1}^j}(v,\eta)\Big|+\Big|\partial_v\widehat{\,F_{1}^j}(v,\eta)\Big|\bigg]\\
&\lesssim \int_{\R^3}e^{-|v-\rho|}e^{\mu\langle k,\eta-\alpha-\beta\rangle^{1/2}}\langle\rho\rangle^{-1}\big|\widehat{\,\,f_{j}}(\eta-\alpha-\beta)\big|e^{-\delta_0\langle\alpha,\beta\rangle^{1/2}}\,d\alpha d\beta  d\rho.
\end{split}
\end{equation}
The desired bound \eqref{BSDG23} for $F_1$ follows from \eqref{BSDG24}.

We finally prove \eqref{BSDG23} for $F_2$. Taking Fourier transform in $v, w$, using the definitions \eqref{BSDG12} and \eqref{BSDG22}, we can bound for $j\in\Z, \xi,\eta\in\R$,
\begin{equation}\label{BSDG27}
\begin{split}
&e^{\mu\langle k,\eta\rangle^{1/2}}|\widehat{\,F_{2}^j}(\xi,\eta)|\lesssim  \int_{\R^5}\big|\widetilde{\,\,\mathcal{G}_{k,j}}(\xi,\alpha_1;\alpha_2)\big|\big|\widetilde{\,\,k^2_{\epsilon,j}}(-\alpha_1,\eta-\alpha_2-\alpha_3+\gamma;\alpha_3)\big|e^{\mu\langle\beta,\alpha_2,\alpha_3\rangle^{1/2}}\\
&\hspace{1.5in}\times e^{-\delta_0\langle\beta\rangle^{1/2}}e^{\mu\langle k,\eta-\beta-\alpha_2-\alpha_3\rangle^{1/2}}\big|\widehat{\,f_j}(\eta-\beta-\alpha_2-\alpha_3)\big|\,d\Xi\\
&\lesssim \int_{\R^5}\frac{e^{-\delta_0\langle\xi+\alpha_1,\alpha_2,\alpha_3,\beta\rangle^{1/2}}}{k^2+\xi^2}\Big|e^{\delta_0\langle\alpha_3\rangle^{1/2}}\widetilde{\,\,k^2_{\epsilon,j}}(-\alpha_1,\eta-\alpha_2-\alpha_3+\gamma;\alpha_3)\Big|\big|\widehat{\,g_j}(\eta-\beta-\alpha_2-\alpha_3)\big|\,d\Xi.
\end{split}
\end{equation}
In the above, we have used the notation $d\Xi=d\alpha_1d\alpha_2d\alpha_3d\beta d\gamma$, the bounds \eqref{BSDG4.005} and the definiton
\begin{equation}\label{BSDG27.5}
\widehat{\,g_j}(\eta):=e^{\mu\langle k,\eta\rangle^{1/2}}\widehat{\,f_j}(\eta), \quad{\rm for}\,\,\eta\in\R.
\end{equation}

 The desired bounds \eqref{BSDG23} for $F_2$ now follow from \eqref{BSDG27} and the bounds \eqref{GKA14.3}.

\end{proof}

To control the commutator terms, we shall need the following estimates.
\begin{lemma}\label{BSDG28}
Fix $\mu\in(0,1/8)$ with $0<\mu\ll \delta_0$. Assume that $k\in\Z\backslash\{0\}$, $\epsilon\in(-1/8,1/8)\backslash\{0\}$. Define for any $h(v,w)\in L^2(\R^2)$,
\begin{equation}\label{BSDG29}
\mathcal{I}(h)(v,w):=\int_{\R^2}\mathcal{G}_{k}(v,v';w)k_{\epsilon}(v',\rho;w)\frac{\partial_\rho B^\ast(\rho+w)}{B^\ast(\rho+w)} \frac{h(\rho,w)}{\rho+i\epsilon^{1/3}} d\rho dv'
\end{equation}
Define also the Fourier multiplier operator $M_\mu$ as follows. For any $f\in L^2(\R)$,
\begin{equation}\label{BSDG30}
\widehat{M_\mu f}(\eta):=e^{\mu\langle k,\eta\rangle^{1/2}}\widehat{f}(\eta),\,\,{\rm for}\,\,\eta\in\R.
\end{equation}
For $h\in L^2(\R^2)$ with $\|\langle k,\partial_v\rangle M_\mu h\|_{L^2(\R^2)}<\infty$, assuming that $M_\mu$ acts on the $w$ variable here and below,  we have the following commutator bounds 
\begin{equation}\label{BSDG31}
\left\|\langle k,\partial_v\rangle\left[M_\mu\big\{\mathcal{I}(h)\big\}(v,w)-\mathcal{I}(M_\mu h)(v,w)\right]\right\|_{L^2(\R^2)}\lesssim \left\|\langle k,\partial_v\rangle\langle\partial_w\rangle^{-1/2} M_\mu\big\{\mathcal{I}(h)\big\}(v,w)\right\|_{L^2(\R^2)}.
\end{equation}
\end{lemma}

\begin{proof}
Using a Gevrey cutoff function $\Psi$ as in \eqref{GKA10} and $\Psi^\ast$ as in \eqref{BSDG13}, we define for $v,w\in\R$, 
\begin{equation}\label{BSDG32}
\mathcal{I}_{N}(h)(v,w):=\int_{\R^2}\mathcal{G}_{k}(v,v';w)k_{\epsilon}(v',\rho;w)(1-\Psi(\rho))\frac{\partial_\rho B^\ast(\rho+w)}{B^\ast(\rho+w)} \frac{h(\rho,w)}{\rho+i\epsilon^{1/3}} d\rho dv',
\end{equation}
\begin{equation}\label{BSDG32.5}
\mathcal{I}_{L}(h)(v,w):=\int_{\R^2}\mathcal{G}_{k}(v,v';w)(1-\Psi^\ast(v'))k_{\epsilon}(v',\rho;w)\Psi(\rho)\frac{\partial_\rho B^\ast(\rho+w)}{B^\ast(\rho+w)} \frac{h(\rho,w)}{\rho+i\epsilon^{1/3}} d\rho dv',
\end{equation}
\begin{equation}\label{BSDG33}
\mathcal{I}_{1}(h)(v,w):=\int_{\R^2}\mathcal{G}_{k}(v,v';w)\Psi^\ast(v')k^1_{\epsilon}(v',\rho;w)\Psi^\ast(\rho)\frac{\partial_\rho B^\ast(\rho+w)}{B^\ast(\rho+w)} \frac{h(\rho,w)}{\rho+i\epsilon^{1/3}} d\rho dv',
\end{equation}
\begin{equation}\label{BSDG34}
\mathcal{I}_{2}(h)(v,w):=\int_{\R^2}\mathcal{G}_{k}(v,v';w)\Psi^\ast(v')k^2_{\epsilon}(v',\rho;w)\Psi^\ast(\rho)\frac{\partial_\rho B^\ast(\rho+w)}{B^\ast(\rho+w)} \frac{h(\rho,w)}{\rho+i\epsilon^{1/3}} d\rho dv'.
\end{equation}
It suffices to prove that for $\ast\in\{N, L,1, 2\}$,
\begin{equation}\label{BSDG35}
\left\|\langle k,\partial_v\rangle\left[M_\mu\big\{\mathcal{I}(h_\ast)\big\}(v,w)-\mathcal{I}(M_\mu h_\ast)(v,w)\right]\right\|_{L^2(\R^2)}\lesssim \left\|\langle k,\partial_v\rangle\langle\partial_w\rangle^{-1/2} M_\mu\big\{\mathcal{I}(h)\big\}(v,w)\right\|_{L^2(\R^2)}.
\end{equation}
For notational convenience, we set for $\eta\in\R$,
\begin{equation}\label{BSDG35.1}
m_\mu(\eta):=e^{\mu\langle k,\eta\rangle^{1/2}}.
\end{equation}

We first prove \eqref{BSDG35} for the case $\ast=N$. Recall that $\partial_vB^\ast$ is compactly supported. Using the bound \eqref{BSDG4.005} and \eqref{GKA13}, we obtain that 
\begin{equation}\label{BSDG36}
\begin{split}
&|k|\left|m_\mu(\eta)\widehat{\,\,\mathcal{I}_{N}(h)}(v,\eta)-\widehat{\,\,\mathcal{I}_{N}(M_\mu h)}(v,\eta)\right|+\left|\partial_v\Big[m_\mu(\eta)\widehat{\,\,\mathcal{I}_{N}(h)}(v,\eta)-\widehat{\,\,\mathcal{I}_{N}(M_\mu h)}(v,\eta)\Big]\right|\\
&\lesssim \int_{\R^3}e^{-\delta_0|v-\rho|}\langle\rho\rangle^{-1}\big|m_\mu(\eta)-m_\mu(\eta-\alpha-\beta-\gamma)\big|\big|\widehat{\,\,h\,\,}(\rho,\eta-\alpha-\beta-\gamma)\big|e^{-\delta_0\langle\alpha,\beta,\gamma\rangle^{1/2}}\,d\alpha d\beta d\gamma d\rho.
\end{split}
\end{equation}
The desired bound \eqref{BSDG35} for $\ast=N$ follows from \eqref{BSDG36}, using also the pointwise inequality for $\alpha,\beta\in\R$,
\begin{equation}\label{BSDG36.1}
|m_\mu(\alpha)-m_\mu(\beta)|\lesssim \langle\alpha,\beta\rangle^{-1/2}|\alpha-\beta|\Big[m_\mu(\alpha)+m_\mu(\beta)\Big].
\end{equation}

The proof of \eqref{BSDG35} for the case $\ast=L$ follows similar lines, using also the pointwise inequality \eqref{BSDG15.001}.

We now turn to the proof of \eqref{BSDG35} for the case $\ast=1$. Using \eqref{BSDG4.005} and \eqref{GKA14.2}, we have 
\begin{equation}\label{BSDG37}
\begin{split}
&|k|\left|m_\mu(\eta)\widehat{\,\,\mathcal{I}_{1}(h)}(v,\eta)-\widehat{\,\,\mathcal{I}_{1}(M_\mu h)}(v,\eta)\right|+\left|\partial_v\Big[m_\mu(\eta)\widehat{\,\,\mathcal{I}_{1}(h)}(v,\eta)-\widehat{\,\,\mathcal{I}_{1}(M_\mu h)}(v,\eta)\Big]\right|\\
&\lesssim \int_{\R^3}e^{-\delta_0|v-\rho|}\langle\rho\rangle^{-1}\big|m_\mu(\eta)-m_\mu(\eta-\alpha-\beta-\gamma)\big|\big|\widehat{\,\,h\,\,}(\rho,\eta-\alpha-\beta-\gamma)\big|e^{-\delta_0\langle\alpha,\beta,\gamma\rangle^{1/2}}\,d\alpha d\beta d\gamma d\rho.
\end{split}
\end{equation}
The desired bound \eqref{BSDG35} for $\ast=1$ then follows from \eqref{BSDG37}.

Finally, we treat the case $\ast=2$. Using \eqref{BSDG4.005} and \eqref{GKA14.3}, we obtain that 
\begin{equation}\label{BSDG38}
\begin{split}
&\left|m_\mu(\eta)\widetilde{\,\,\mathcal{I}_{2}(h)}(\xi,\eta)-\widetilde{\,\,\mathcal{I}_{2}(M_\mu h)}(\xi,\eta)\right|\\
&\lesssim  \int_{\R^5}\big|\widetilde{\,\,\mathcal{G}_{k,0}}(\xi,\alpha_1;\alpha_2)\big|\big|\widetilde{\,\,k^2_{\epsilon,0}}(-\alpha_1,-\gamma-\beta-\alpha_4;\alpha_3)\big|\\
&\times e^{-\delta_0\langle\beta\rangle^{1/2}}|m_\mu(\eta-\beta-\alpha_2-\alpha_3)-m_\mu(\eta)|\big|\widehat{\,\,h\,\,}(\alpha_4,\eta-\beta-\alpha_2-\alpha_3)\big|\,d\Xi\\
&\lesssim \int_{\R^5}\frac{e^{-\delta_0\langle\xi+\alpha_1,\alpha_2,\alpha_3,\beta\rangle^{1/2}}}{k^2+\xi^2}\left|e^{\delta_0\langle\alpha_3\rangle^{1/2}}\widetilde{\,\,k^2_{\epsilon,j}}(-\alpha_1,-\gamma-\beta-\alpha_4;\alpha_3)\right|\\
&\times|m_\mu(\eta-\beta-\alpha_2-\alpha_3)-m_\mu(\eta)|\big|\widehat{\,\,h\,\,}(\alpha_4,\eta-\beta-\alpha_2-\alpha_3)\big|\,d\Xi,
\end{split}
\end{equation}
where $d\Xi=d\alpha_1d\alpha_2d\alpha_3d\alpha_4d\beta d\gamma$. The desired bound \eqref{BSDG35} for $\ast=2$ then follows from \eqref{BSDG37}.

\end{proof}

We are now ready to prove Proposition \ref{int14} (which we recall below) on the main Gevrey bounds for the renormalized spectral density function. 
\begin{proposition}\label{BSDG40}
There exists $\nu_0\in(0,1/8)$ sufficiently small, such that the following statement holds. For $k\in\Z\backslash\{0\}$, $\nu\in(0,\nu_0)$ and $0<\delta\ll \delta_0$, the profile for the spectral density function, $\Theta_{k,\nu}(v,w)$, satisfies the bounds 
\begin{equation}\label{BSDG42}
\Big\|(|k|+|\xi|)\widetilde{\,\,\Theta_{k,\nu}}(\xi,\eta)\Big\|_{L^2(\R^2)}\lesssim \Big\|\widehat{\,\,F_{0k}}(\eta)\Big\|_{L^2(\R)}.
\end{equation}
and
\begin{equation}\label{BSDG42.5}
\Big\|(|k|+|\xi|)e^{\delta \langle k,\eta\rangle^{1/2}}\widetilde{\,\,\Theta_{k,\nu}}(\xi,\eta)\Big\|_{L^2(\R^2)}\lesssim \Big\|e^{\delta\langle k,\eta\rangle^{1/2}}\widehat{\,\,F_{0k}}(\eta)\Big\|_{L^2(\R)}.
\end{equation}
\end{proposition}

\begin{proof}
We can reformulate equation \eqref{BSDG1} in the integral form as
\begin{equation}\label{BSDG43}
\Theta_{k,\nu}(v,w)+i\int_{\R^2}\mathcal{G}_{k}(v,v';w)k_{\epsilon}(v',\rho;w)\frac{\partial_\rho B^\ast(\rho+w)}{B^\ast(\rho+w)} \frac{\Theta_{k,\nu}(\rho,w)}{\rho+i\epsilon^{1/3}} d\rho dv'=F_{k,\nu}(v,w),\quad {\rm for}\,\,v\in\R,
\end{equation}
for $v, w\in\R$, where 
\begin{equation}\label{BSDG44}
F_{k,\nu}(v,w):=\int_{\R^2}\mathcal{G}_k(v,v';w)k_\epsilon(v',\rho;w)\frac{1}{(B^\ast(\rho+w))^2}\frac{F_{0k}(\rho+w)}{\rho+i\epsilon^{1/3}} d\rho dv'.
\end{equation}
By Lemma \ref{BSDG19}, we have the bounds 
\begin{equation}\label{BSDG45}
\|\langle k,\partial_v\rangle F_{k,\nu}\|_{L^2(\R^2)}\lesssim \|F_{0k}\|_{L^2(\R)},
\end{equation}
and
\begin{equation}\label{BSDG46}
\|\langle k,\xi\rangle e^{\delta \langle k,\eta\rangle^{1/2}} \widetilde{\,\,F_{k,\nu}}(\xi,\eta)\|_{L^2(\R^2)}\lesssim \|e^{\delta \langle k,\eta\rangle^{1/2}}\widetilde{\,\,F_{0k}}(\xi)\|_{L^2(\R)}.
\end{equation}
The desired bound \eqref{BSDG42} follows from \eqref{BSDG43}, Proposition \ref{BSDG2} and the bound \eqref{BSDG45}. 

To establish \eqref{BSDG42.5}, we apply the Fourier multiplier operator $M_\mu$ to \eqref{BSDG43} (acting on $w$), and use the commutator estimates. More precisely, using \eqref{BSDG43} we get 
\begin{equation}\label{BSDG47}
\begin{split}
&M_\mu\big[\Theta_{k,\nu}(v,\cdot)\big](w)+i\int_{\R^2}\mathcal{G}_{k}(v,v';w)k_{\epsilon}(v',\rho;w)\frac{\partial_\rho B^\ast(\rho+w)}{B^\ast(\rho+w)} \frac{M_\mu\big[\Theta_{k,\nu}(\rho,\cdot)\big](w)}{\rho+i\epsilon^{1/3}} d\rho dv'\\
&=M_\mu\big[F_{k,\nu}(v,\cdot)\big](w)-i\Big\{M_\mu\big[\mathcal{I}(\Theta_{k,\nu})(v,\cdot)\big](w)-\mathcal{I}\big(M_\mu\big[\Theta_{k,\nu}(v,\cdot)\big](w)\big)\Big\},
\end{split}
\end{equation}
for $v,w\in\R$. Lemma \ref{BSDG28} implies that
\begin{equation}\label{BSDG48}
\begin{split}
&\left\|\langle k,\partial_v\rangle \Big\{M_\mu\big[\mathcal{I}(\Theta_{k,\nu})(v,\cdot)\big](w)-\mathcal{I}\big(M_\mu\big[\Theta_{k,\nu}(v,\cdot)\big](w)\big)\Big\}\right\|_{L^2(\R^2)}\\
&\lesssim \gamma\left\|\langle k,\partial_v\rangle M_\mu \Theta_{k,\nu}\right\|_{L^2(\R^2)}+C_\gamma \left\|\langle k,\partial_v\rangle\Theta_{k,\nu}\right\|_{L^2(\R^2)},
\end{split}
\end{equation}
for any $\gamma\in(0,1)$ and a suitable $C_\gamma\in(0,\infty)$, by dividing into high frequencies and low frequencies in $w$. Then it follows from Proposition \ref{BSDG2}, Lemma \ref{BSDG19} and the bound \eqref{BSDG48} that
\begin{equation}\label{BSDG49}
\left\|\langle k,\partial_v\rangle M_\mu \Theta_{k,\nu}\right\|_{L^2(\R^2)}\lesssim \left\|M_\mu F_{0k}\right\|_{L^2(\R)}+ \gamma\left\|\langle k,\partial_v\rangle M_\mu \Theta_{k,\nu}\right\|_{L^2(\R^2)}+C_\gamma \left\|\langle k,\partial_v\rangle\Theta_{k,\nu}\right\|_{L^2(\R^2)}.
\end{equation}
 The desired bound \eqref{BSDG42.5} then follows from \eqref{BSDG49} by choosing $\gamma\in(0,1)$ sufficiently small.

\end{proof}


\section{The spectral representation formula}\label{sec:trf}
In this section, we justify the representation formula \eqref{int6} for $\omega_{0k}\in C_c^\infty(\R)$. The basic idea is to use the general theory of sectorial operators, and we refer to section 4, Chapter II of \cite{Engel} for the well known approach. The main properties we need to establish are the following.
\begin{proposition}\label{rep1}
Assume that $k\in\Z\backslash\{0\}$ and $\omega_{0k}\in C_c^\infty(\R)$. There exists $\nu_\ast\in(0,1)$, such that the following statement holds. Define
\begin{equation}\label{rep2}
\Sigma:=\big\{\lambda\in\C: \,k\, \Re \lambda\leq 0\big\}.
\end{equation}
For any $\lambda\in\Sigma$ and $\nu\in(0,\nu_\ast)$, the following system of equations for $y\in\R$,
\begin{equation}\label{rep3}
\begin{split}
\frac{\nu}{k}\partial_y^2h(y,\lambda)-(ib(y)-\lambda)h(y,\lambda)+ib''(y)\varphi(y,\lambda)&=\omega_{0k}(y),\\
(-k^2+\partial_y^2)\varphi(y,\lambda)&=h(y,\lambda),
\end{split}
\end{equation}
is solvable for a unique solution $h(y,\lambda)\in H^2(\R)$ which is holomorphic for $\lambda$ in the interior of $\Sigma$ and continuous on $\Sigma$. In addition, we have the bounds 
\begin{equation}\label{rep4}
\big\|h(\cdot,\lambda)\big\|_{L^2(\R)}\lesssim_{\nu,\omega_{0k}}\frac{1}{\langle \lambda\rangle}, 
\end{equation}
and
\begin{equation}\label{rep4.5}
 \big\|\partial_\lambda h(\cdot,\lambda)\big\|_{L^2(\R)}\lesssim_{\nu,\omega_{0k}}\frac{1}{\langle \lambda\rangle^2}
\end{equation}

\end{proposition}
 
 We note that in the bounds \eqref{rep4}-\eqref{rep4.5} we allow the implied constants to depend on the small viscosity $\nu>0$ and an unspecified norm on $\omega_{0k}$, which is not a problem for us as we only use these bounds in a qualitative way, to justify the identity \eqref{int6}.

\begin{proof}
Assume for the sake of concreteness $k>0$ and correspondingly $\Re \lambda\leq0$. By Propositions \ref{LAPM1} and \ref{LAP22}, the equation \eqref{rep3} is solvable and we have the bound 
\begin{equation}\label{rep5}
\|\varphi(\cdot,\lambda)\|_{H^1(\R)}\lesssim_{\omega_{0k}}1,\quad \|h(\cdot,\lambda)\|_{H^1(\R)}\lesssim_{\nu,\omega_{0k}}1.
\end{equation}
Using Lemma \ref{GKA1} and the assumption on the compact support of $b''$ and $\omega_{0k}$, we obtain that
\begin{equation}\label{rep6}
 \|h(\cdot,\lambda)\|_{H^1(\R)}\lesssim_{\nu,\omega_{0k}}\langle\lambda\rangle^{-1},
\end{equation}
and for $|y|$ sufficiently large (so that $y$ is outside the support of $b''$ and $\omega_{0k}$), $|\lambda|\gg1$,
\begin{equation}\label{rep7}
|h(y,\lambda)|\lesssim_{\nu,\omega_{0k}} e^{-\delta_0\langle\epsilon^{-1/3}\lambda\rangle^{1/2}\epsilon^{-1/3}\langle y\rangle}.
\end{equation}
Notice that $\partial_\lambda \varphi(y,\lambda)$ satisfies the equation for $y\in\R$ and $\lambda\in\Sigma$,
\begin{equation}\label{rep8}
\begin{split}
\frac{\nu}{k}\partial_y^2\partial_\lambda h(y,\lambda)-(ib(y)-\lambda)\partial_\lambda h(y,\lambda)+ib''(y)\partial_\lambda\varphi(y,\lambda)&=h(y,\lambda),\\
(-k^2+\partial_y^2)\partial_\lambda\varphi(y,\lambda)&=\partial_\lambda h(y,\lambda).
\end{split}
\end{equation}
The desired bounds \eqref{rep4}-\eqref{rep4.5} follow from \eqref{rep6}, \eqref{rep8}, the limiting absorption principle, and \eqref{rep7}. The Proposition is now proved.
\end{proof}


\section{Decay of semigroups using resolvent bounds}\label{sec:dsr}

In this section we prove a quantitative bound on the decay of bounded semigroups using resolvent estimates, which is of independent interest. The main theorem and its proof are modifications of a similar result in Wei \cite{Wei3} (see also the earlier work of Helffer and Sj\"ostrand \cite{Helffer2}, and the more recent \cite{Helffer3} which contained a more general form the estimates below, with sharp constants). The key difference with \cite{Wei3} is that we do not assume the operator which generates the semigroup to be ``accretive", and instead require the semigroup to be bounded to begin with. The formulation allows a wider range of applications than semigroups generated by accretive operators.

\begin{theorem}\label{DSR0}
Assume that $\mu>0, M\ge1$. Suppose that $X$ is a separable Hilbert space, and the linear operator $A: \mathcal{D}(A)\subseteq X\to X$ is densely defined and closed.  Assume that $A$ generates the bounded semigroup $e^{At}$, $t\ge0$ with the bound 
\begin{equation}\label{DSR1}
\left\|e^{At}\right\|_{X\to X}\leq M. 
\end{equation}
 Assume that for $\lambda\in\R$, $i\lambda-A: \mathcal{D}(A)\to X$ is invertible and we have the resolvent bound 
\begin{equation}\label{DSR2}
\left\|(i\lambda-A)^{-1}\right\|_{X\to X}\leq \mu^{-1}.
\end{equation}
Then we have the decay estimates for the semigroup $e^{At}, t\ge0$,
\begin{equation}\label{DSR3}
\left\|e^{At}\right\|_{X\to X}\leq C_0 M^2 e^{-\mu t},
\end{equation}
for a universal constant $C_0>0$.

\end{theorem}

\begin{proof}
By renormalization, we can assume that $\mu=1$. Assume that $C_0\ge10$ is sufficiently large, then in view of the bound \eqref{DSR1} it suffices to consider $t\ge20$. We choose a positive number $L\ge2$ whose value is to be determined below. For $t\ge5$, fix $\varphi\in C^1([0,t+L])$ with $\varphi(0)=\varphi(t+L)=0$. For any $x\in \mathcal{D}(A)$ with $\|x\|_X=1$, set 
\begin{equation}\label{DSR4}
g(t):=e^{At}x.
\end{equation}
Direct computation shows $\varphi(t)g(t)$ satisfies the equation
\begin{equation}\label{DSR5}
\partial_t\big(\varphi(t)g(t)\big)-A\big(\varphi(t)g(t)\big)=\varphi'(t)g(t).
\end{equation}
Taking Fourier transform in $t\in\R$, we obtain that 
\begin{equation}\label{DSR6}
(i\tau-A)\widehat{\,\varphi g\,}(\tau)=\widehat{\,\varphi'g\,}(\tau),\qquad \tau\in\R.
\end{equation}
From \eqref{DSR6} we conclude that for $\tau\in\R$
\begin{equation}\label{DSR7}
\left\|\widehat{\,\varphi g\,}(\tau)\right\|_{X}=\left\|(i\tau-A)^{-1}\widehat{\,\varphi' g\,}(\tau)\right\|_X\leq \left\|\widehat{\,\varphi' g\,}(\tau)\right\|_{X},
\end{equation}
which implies (by Parseval's identity) that
\begin{equation}\label{DSR8}
\int_0^{t+L}|\varphi(s)|^2\|g(s)\|_X^2\,ds\leq \int_0^{t+L}|\varphi'(s)|^2\|g(s)\|_X^2\,ds.
\end{equation}
The inequality \eqref{DSR8}, which is analogous to the entanglement inequality \eqref{BKGL32}, is very powerful in the sense that by appropriately choosing $\varphi$, it allows to connect the bounds on $g(s)$ for $s$ near $t$
with bounds for $g(s)$ near the initial time $s=0$.

Define $\varphi(s)$ for $s\in[0,t+L]$ as 
\begin{equation}\label{DSR9}
\varphi(s)=s{\bf 1}_{[0,1]}(s)+e^{s-1}{\bf 1}_{[1,t]}(s)+e^{t-1}\big(1-\frac{s-t}{L}\big){\bf 1}_{[t,t+L]}(s).
\end{equation}
By \eqref{DSR8}, we have
\begin{equation}\label{DSR10}
\int_t^{t+L}\big(|\varphi(s)|^2-|\varphi'(s)|^2\big)\|g(s)\|_X^2\,ds\leq M^2\int_0^1(1-s^2)\,ds \leq M^2.
\end{equation}
For $s\in[t,t+L]$, 
\begin{equation}\label{DSR11}
|\varphi(s)|^2-|\varphi'(s)|^2=e^{2(t-1)}\Big[\big(1-\frac{s-t}{L}\big)^2-\frac{1}{L^2}\Big],
\end{equation}
therefore using the semigroup bound \eqref{DSR1}, we have the bound 
\begin{equation}\label{DSR12}
\begin{split}
&\int_t^{t+L}\big(|\varphi(s)|^2-|\varphi'(s)|^2\big)\|g(s)\|_X^2\,ds\\
&\ge \int_t^{t+1}M^{-2}\|g(t+1)\|^2_X\big(|\varphi(s)|^2-|\varphi'(s)|^2\big)\,ds-\int_{t+L-1}^{t+L}\big||\varphi(s)|^2-|\varphi'(s)|^2\big|M^2\|g(t+1)\|_X^2\,ds\\
&\ge \|g(t+1)\|_X^2e^{2(t-1)}\Big\{\frac{1}{M^2}\Big((1-\frac{1}{L})^2-\frac{1}{L^2}\Big)-\frac{M^2}{L^2}\Big\}.
\end{split}
\end{equation}
Taking $L\ge10M^2$, then we obtain that
\begin{equation}\label{DSR13}
\int_t^{t+L}\big(|\varphi(s)|^2-|\varphi'(s)|^2\big)\|g(s)\|_X^2\,ds\ge \frac{1}{20M^2}\|g(t+1)\|_X^2e^{2(t-1)},
\end{equation}
which together with \eqref{DSR10} implies that
\begin{equation}\label{DSR14}
\|g(t+1)\|_X^2e^{2(t-1)}\leq 20 M^4.
\end{equation}
Consequently we obtain that
\begin{equation}\label{DSR15}
\|g(t+1)\|_X^2\leq 20e^2M^4e^{-2t}.
\end{equation}
The theorem is proved.

\end{proof}


\section{Proof of the main theorem}\label{sec:pmt}
In this section we give the proof of Theorem \ref{intM1}, using Proposition \ref{int14}. Denote $\epsilon=\nu/k$. For the sake of concreteness, we assume that $\epsilon>0$, as the other case is completely analogous.

We begin with the proof of \eqref{intM3}-\eqref{intM4}. Due to the similarity in the argument, we shall focus only on the proof of \eqref{intM4} which is slightly more complicated. We use the representation formula \eqref{int13.1}, the definition \eqref{GKA6} for the kernel $k_\epsilon$ and equation \eqref{intP8} to obtain that
\begin{equation}\label{pmt3}
\begin{split}
F_k(t,v)=&-\frac{i}{2\pi}e^{-\nu k^2t}\int_{\R^2} e^{ik(v-w)t}k_\epsilon(v-w,\rho;w)(\partial_\rho B^\ast/B^\ast)(\rho+w)\frac{\Theta_{k,\nu}(\rho,w)}{\rho+i\epsilon^{1/3}}\,d\rho dw\\
&+\frac{1}{2\pi}e^{-\nu k^2t}\int_{\R^2} e^{ik(v-w)t}k_\epsilon(v-w,\rho;w)\frac{1}{(B^\ast(\rho+w))^2}\frac{F_{0k}(\rho+w)}{\rho+i\epsilon^{1/3}}\,d\rho dw\\
=&-\frac{i}{2\pi}e^{-\nu k^2t}\int_{\R^2} e^{ikw't}k_\epsilon(w',\rho;v-w')(\partial_\rho B^\ast/B^\ast)(\rho+v-w')\frac{\Theta_{k,\nu}(\rho,v-w')}{\rho+i\epsilon^{1/3}}\,d\rho dw'\\
&+\frac{1}{2\pi}e^{-\nu k^2t}\int_{\R^2} e^{ikw't}k_\epsilon(w',\rho;v-w')\frac{1}{(B^\ast(\rho+v-w'))^2}\frac{F_{0k}(\rho+v-w')}{\rho+i\epsilon^{1/3}}\,d\rho dw', 
\end{split}
\end{equation}
for $v\in\R$ and $t\ge0$. Recall the Gevrey cutoff function $\Psi$ from \eqref{GKA10} and $\Psi^\ast$ from \eqref{BSDG13}. Define the localization functions for $w,\rho\in\R$,
\begin{equation}\label{pmt3.1}
q_N(w,\rho):=1-\Psi(\rho),\quad q_L(w,\rho):=(1-\Psi^\ast(w))\Psi(\rho),\quad q(w,\rho):=\Psi^\ast(w)\Psi(\rho).
\end{equation}
We define for $\dagger\in\{N,L\}$ and $v\in\R, t\ge0$,
\begin{equation}\label{pmt5}
S_k^\dagger(t,v):=-\frac{ie^{-\nu k^2t}}{2\pi}\int_{\R^2} e^{ikwt}k_\epsilon(w,\rho;v-w)(\partial_\rho B^\ast/B^\ast)(\rho+v-w)q_\dagger(w,\rho)\frac{\Theta_{k,\nu}(\rho,v-w)}{\rho+i\epsilon^{1/3}}\,d\rho dw,
\end{equation}
and for $\dagger\in\{1,2\}$ and $v\in\R, t\ge0$
\begin{equation}\label{pmt6}
S_k^\dagger(t,v):=-\frac{ie^{-\nu k^2t}}{2\pi}\int_{\R^2} e^{ikwt}k^\dagger_\epsilon(w,\rho;v-w)(\partial_\rho B^\ast/B^\ast)(\rho+v-w)q(w,\rho)\frac{\Theta_{k,\nu}(\rho,v-w)}{\rho+i\epsilon^{1/3}}\,d\rho dw.
\end{equation}
Define also for $v\in\R, t\ge0$,
\begin{equation}\label{pmt7}
I_k^N(t,v):=\frac{1}{2\pi}e^{-\nu k^2t}\int_{\R^2} e^{ikwt}k_\epsilon(w,\rho;v-w)\frac{1-\Psi(\rho-w)}{(B^\ast(\rho+v-w))^2}\frac{F_{0k}(\rho+v-w)}{\rho+i\epsilon^{1/3}}\,d\rho dw,
\end{equation}
and for $v\in\R, t\ge0$
\begin{equation}\label{pmt8}
I_k^L(t,v):=\frac{1}{2\pi}e^{-\nu k^2t}\int_{\R^2} e^{ikwt}k_\epsilon(w,\rho;v-w)\frac{\Psi(\rho-w)}{(B^\ast(\rho+v-w))^2}\frac{F_{0k}(\rho+v-w)}{\rho+i\epsilon^{1/3}}\,d\rho dw.
\end{equation}
We can decompose
\begin{equation}\label{pmt9}
F_k(t,v)=\sum_{\dagger\in\{N,L,1,2\}}S_k^\dagger(t,v)+\sum_{\dagger\in\{N,L\}}I_k^\dagger(t,v).
\end{equation}
It suffices to prove that for $h_k\in\{S_k^\dagger: \dagger\in\{N,L,1,2\}\}$ and $h_k\in\{I_k^N, I_k^L\}$,
\begin{equation}\label{pmt10}
\left\|e^{\delta\langle k,\xi\rangle^{1/2}}\widehat{\,\,h_k}(t,\xi)\right\|_{L^2(\R)}\lesssim e^{-\nu k^2t}\left\|e^{\delta\langle k,\xi\rangle^{1/2}}\widehat{\,\,F_{0k}}(\xi)\right\|_{L^2(\R)}.
\end{equation}
For notational conveniences, we denote for $w,\rho\in\R$,
\begin{equation}\label{pmt10.1}
E_{\epsilon,\delta_0}(\rho,w):=e^{-\delta_0\langle\epsilon^{-1/3}\rho,\epsilon^{-1/3}w\rangle^{1/2}\epsilon^{-1/3}|w-\rho|}
\end{equation}
For $h_k\in\{S_k^N, S_k^L, S_k^1\}$, we can bound, using \eqref{GKA13}, \eqref{GKA14.2}, and the fact that $\partial_vB^\ast$ is compactly supported, that for $t\ge0$ and $\xi\in\R$,
\begin{equation}\label{pmt11}
\begin{split}
|\widehat{\,\,h_k}(t,\xi)|\lesssim e^{-\nu k^2t}\int_{\R^3}\epsilon^{-1/3}\langle\epsilon^{-1/3}\rho,\epsilon^{-1/3}w\rangle^{1/2}E_{\epsilon,\delta_0}(\rho,w)e^{-\delta_0\langle\alpha\rangle^{1/2}}\frac{|\widehat{\,\,\,\Theta_{k,\nu}}(\rho,\xi-\alpha)|}{\langle\rho\rangle} d\rho dw d\alpha.
\end{split}
\end{equation}
The desired bounds \eqref{pmt10} for $h_k\in\{S_k^N, S_k^L, S_k^1\}$ then follow from \eqref{pmt11} and the bounds \eqref{int15}.

For $h_k=S_k^2$, we take Fourier transform and obtain that
\begin{equation}\label{pmt12}
|\widehat{\,\,h_k}(t,\xi)|\lesssim e^{-\nu k^2t}\int_{\R^4}\left|\widetilde{\,\,k_{\epsilon,0}^2}(\xi-kt,\alpha_1;\beta)\right||\widetilde{(B^\ast\partial B^\ast)}(\gamma)|\left|\widetilde{\,\,\Theta_{k,\nu}}(-\alpha_1-\alpha_2-\gamma,\xi-\beta-\gamma)\right| d\Xi,
\end{equation}
where $d\Xi=d\alpha_1 d\alpha_2d\beta d\gamma$. The desired bound \eqref{pmt10} then follows from the bounds \eqref{int15} and \eqref{GKA14.3}.

We now turn to the terms $h_k\in\{I_k^N, I_k^L\}$. Fix a nonnegative Gevrey-regular partition of unity function $\Psi^p\in C_0^\infty(-2,2)$. More precisely we impose
\begin{equation}\label{pmt13}
\sum_{j\in\Z}\Psi^p(v-j)\equiv1,\quad{\rm for}\,\,v\in\R,
\end{equation}
and
\begin{equation}\label{pmt14}
\sup_{\xi\in\R}\Big[e^{\langle\xi\rangle^{3/4}}|\widehat{\,\,\Psi^p}(\xi)|\Big]\lesssim1.
\end{equation}
Denote for $v\in\R$ and $j\in\Z$, 
\begin{equation}\label{pmt15}
F_{0k}^j(v):=F_{0k}(v)\Psi^{\ast\ast}(v-j)/(B^\ast(v))^2.
\end{equation}
We first treat the case when $h_k=I_k^N$. Define for $j\in\Z$, $t\ge0$ and $v\in\R$,
\begin{equation}\label{pmt16}
\begin{split}
&2\pi e^{\nu k^2t}I_{k,j}^N(t,v):=\int_{\R^2} e^{ikwt}k_\epsilon(w,\rho;v-w)\Psi^p(v-j)\frac{1-\Psi(\rho-w)}{(B^\ast(\rho+v-w))^2}\frac{F_{0k}(\rho+v-w)}{\rho+i\epsilon^{1/3}}\,d\rho dw\\
&=\sum_{\ell\in\Z}\int_{\R^2} e^{ikwt}k_\epsilon(w,\rho;v-w)\Psi^p(v-j)\big[1-\Psi(\rho-w)\big]\Psi(\rho-w-\ell)\frac{F^{j+\ell}_{0k}(\rho+v-w)}{\rho+i\epsilon^{1/3}}\,d\rho dw.
\end{split}
\end{equation}
Taking Fourier transform in $v$ in \eqref{pmt16} and using the bound \eqref{GKA13}, we obtain that for $\xi\in\R$ and $t\ge0$,
\begin{equation}\label{pmt17}
\begin{split}
&e^{-\nu k^2t}\big|\widehat{\,\,I_{k,j}^N}(t,\xi)\big|\\
&\lesssim \sum_{\ell\in\Z}\int_{\R^3}|\widehat{\,\,k_{\epsilon,j+\ell-\rho}}(w,\rho;\alpha)|\big[1-\Psi(\rho-w)\big]\Psi(\rho-w-\ell)\frac{|\widehat{\,\,F_{0k}^{j+\ell}}(\xi-\alpha)|}{|\rho+i\epsilon^{1/3}|} d\rho dw d\alpha\\
&\lesssim  \sum_{\ell\in\Z\backslash\{0\}}\int_{\R^3}e^{-(\delta_0/2)\langle\epsilon^{-1/3}w,\epsilon^{-1/3}\rho\rangle^{1/2}\epsilon^{-1/3}-\delta_0\langle\alpha\rangle^{1/2}}{\bf 1}_{[-10,10]}(\rho-w-\ell)\big|\widehat{\,\,F_{0k}^{j+\ell}}(\xi-\alpha)\big|\,d\rho dw d\alpha
\end{split}
\end{equation}
The desired bound \eqref{pmt10} then follow from \eqref{pmt17}.

Finally we treat the most difficult case when $h_k=I_k^L$. We need a more precise understanding of the singularity of $k_\epsilon(v,\rho;w)$. Decompose for $v,\rho,w\in\R$,
\begin{equation}\label{pmt18}
k_\epsilon(v,\rho;w):=m_\epsilon(v,\rho;w)+n_\epsilon(v,\rho;w),
\end{equation}
where $m_\epsilon(v,\rho;w)$ satisfies for $v,\rho,w\in\R$,
\begin{equation}\label{pmt19}
\begin{split}
&\Big[\epsilon \partial_v^2+\epsilon\frac{\partial_{\rho}B^\ast(\rho+w)}{B^\ast(\rho+w)}\partial_v-i\frac{v}{(B^\ast(\rho+w))^2}\Big]m_\epsilon(v,\rho;w)\\
&=\Big[\epsilon\frac{\partial_{\rho}B^\ast(\rho+w)}{B^\ast(\rho+w)}\partial_v-\epsilon\frac{\partial_{\rho}B^\ast(v+w)}{B^\ast(v+w)}\partial_v-i\frac{v}{(B^\ast(\rho+w))^2}+i\frac{v}{(B^\ast(v+w))^2}\Big]k_\epsilon(v,\rho;w),
\end{split}
\end{equation}
and $n_\epsilon(v,\rho;w)$ satisfies for $v,\rho,w\in\R$
\begin{equation}\label{pmt20}
\begin{split}
\Big[\epsilon \partial_v^2+\epsilon\frac{\partial_{\rho}B^\ast(\rho+w)}{B^\ast(\rho+w)}\partial_v-i\frac{v}{(B^\ast(\rho+w))^2}\Big]n_\epsilon(v,\rho;w)=(\rho+i\epsilon^{1/3})\delta(v-\rho).
\end{split}
\end{equation}
For any $w_0\in\R$ and $v,\rho,w\in\R$, we set
\begin{equation}\label{pmt21}
m_{\epsilon,w_0}(v,\rho;w):=m_\epsilon(v,\rho;w)\Psi(w-w_0).
\end{equation}
It follows from the bound \eqref{GKA13} that for $v,\rho,w,w_0\in\R$ and $\xi\in\R$,
\begin{equation}\label{pmt22}
\big|\widehat{\,\,m_{\epsilon,w_0}}(v,\rho;\xi)\big|\lesssim \epsilon^{-1/3}|v-\rho|\langle\epsilon^{-1/3}\rho\rangle^{1/2}e^{-\delta_0\langle\epsilon^{-1/3}\rho,\epsilon^{-1/3}v\rangle^{1/2}\epsilon^{-1/3}|v-\rho|} e^{-\delta_0\langle\xi\rangle^{1/2}}.
\end{equation}
To study the property of $n_\epsilon$, we take Fourier transform in $v$ in the equation \eqref{pmt20} and obtain that for $\rho,w,\xi\in\R$,
\begin{equation}\label{pmt22.1}
\Big[-\epsilon\xi^2+i\epsilon\xi\frac{\partial_\rho B^\ast(\rho+w)}{B^\ast(\rho+w)}+\frac{\partial_\xi}{(B^\ast(\rho+w))^2}\Big]\widehat{\,\,n_\epsilon}(\xi,\rho;w)=(\rho+i\epsilon^{1/3})e^{-i\rho\xi}.
\end{equation}
Therefore, recalling the definition \eqref{GKA32.1}, we get that for $\rho,w,\xi\in\R$,
\begin{equation}\label{pmt22.2}
\begin{split}
\widehat{\,\,n_\epsilon}(\xi,\rho;w)&=-(\rho+i\epsilon^{1/3})\int_\xi^\infty e^{-(\epsilon/3)(\gamma^3-\xi^3)(B^\ast(\rho+w))^2+i(\epsilon/2)(\gamma^2-\xi^2)(B^\ast\partial B^\ast)(\rho+w)}e^{-i\rho\gamma}\,d\gamma\\
&=-(\rho+i\epsilon^{1/3})\int_\xi^\infty E_{xp}(\gamma,\xi,\rho+w)e^{-i\rho\gamma}\,d\gamma.
\end{split}
\end{equation}
Define for $q\in\{m,n\}$,
\begin{equation}\label{pmt23}
I_{k}^{q}(t,v):=\frac{e^{-\nu k^2t}}{2\pi}\int_{\R^2} e^{ikwt}q_\epsilon(w,\rho;v-w)\Psi(\rho-w)\frac{F_{0k}(\rho+v-w)}{\rho+i\epsilon^{1/3}}\,d\rho dw.
\end{equation}
We decompose for $t\ge0, v\in\R$,
\begin{equation}\label{pmt24}
I_k^L(t,v)=I_k^m(t,v)+I_k^n(t,v),
\end{equation}
and set for $j\in\Z$,
\begin{equation}\label{pmt25}
I_{k,j}^m(t,v):=I_k^m(t,v)\Psi(v-j),\quad I_{k,j}^n(t,v):=I_k^n(t,v)\Psi(v-j).
\end{equation}
It suffices to prove \eqref{pmt10} for $h_k\in\{I_k^m,I_k^n\}$.

Taking Fourier transform in $v$ in \eqref{pmt23} and using the estimates \eqref{pmt22}, we obtain that for $t\ge0, \xi\in\R$ and $j\in\Z$,
\begin{equation}\label{pmt26}
\begin{split}
&e^{\nu k^2t}\big|\widehat{\,\,I_{k,j}^m}(t,\xi)\big|\lesssim\int_{\R^3}\big|\widehat{\,\,m_{\epsilon,j-\rho}}(w,\rho;\alpha)\big|\Psi(\rho-w)\frac{\big|\widehat{\,\,F_{0k}^j}(\xi-\alpha)\big|}{|\rho+i\epsilon^{1/3}|}\,d\rho dw d\alpha\\
&\lesssim \int_{\R^3}\epsilon^{-2/3}|\rho-w|\langle\epsilon^{-1/3}\rho\rangle^{-1/2}\Psi(\rho-w)E_{\epsilon,\delta_0}(\rho,w)e^{-\delta_0\langle\alpha\rangle^{1/2}}\big|\widehat{\,\,F_{0k}^j}(\xi-\alpha)\big| \,d\rho dw d\alpha.
\end{split}
\end{equation}
It follows from \eqref{pmt26} that for $t\ge0, \xi\in\R$ and $j\in\Z$, 
\begin{equation}\label{pmt28}
\big|\widehat{\,\,I_{k,j}^m}(t,\xi)\big|\lesssim e^{-\nu k^2t}\int_{\R^2}\langle\epsilon^{-1/3}\rho\rangle^{-3/2}e^{-\delta_0\langle\alpha\rangle^{1/2}}\big|\widehat{\,\,F_{0k}^j}(\xi-\alpha)\big| \,d\rho  d\alpha.
\end{equation}
The desired bounds follow from \eqref{pmt28} in the case of $I_k^m$. 

We finally turn to the term $I^n_k$, which (using \eqref{pmt22.2}) can be reformulated as
\begin{equation}\label{pmt29}
\begin{split}
&2\pi e^{\nu k^2t}I_{k,j}^n(t,v)=\int_{\R^4}\widehat{\,\,n_\epsilon}(\alpha,\rho;v-w)e^{ikwt+i\alpha w}\frac{\Psi(\rho-w)}{\rho+i\epsilon^{1/3}}\widehat{\,\,F_{0k}^j}(\beta)e^{i\beta(\rho+v-w)}\,d\rho dw d\alpha d\beta\\
&=-\int_{\R^4}\int_{\alpha}^{\infty}\Psi^{\ast}(v-j)E_{xp}(\gamma,\alpha,\rho+v-w)e^{-i\rho\gamma}e^{ikwt+i\alpha w}\Psi(\rho-w)\widehat{\,\,F_{0k}^j}(\beta)e^{i\beta(\rho+v-w)}\,d\Xi,
\end{split}
\end{equation}
where $d\Xi=d\gamma d\rho dw d\alpha d\beta$. Define for $j\in\Z$ and $\gamma,\alpha,w\in\R$,
\begin{equation}\label{pmt30}
E_{xp}^j(\gamma,\alpha,w):=E_{xp}(\gamma,\alpha,w)\Psi^{\ast\ast}(w-j).
\end{equation}
We have the bounds for $\gamma,\alpha,a\in\R$,
\begin{equation}\label{pmt31}
\Big|\int_\R E_{xp}^j(\gamma,\alpha,w)e^{-iwa}\,dw\Big|\lesssim e^{-\delta_0\epsilon(\gamma^3-\alpha^3)}e^{-\delta_0\langle a\rangle^{1/2}}.
\end{equation}
It follows from \eqref{pmt29} and \eqref{pmt31} that for $t\ge0, \xi\in\R$,
\begin{equation}\label{pmt32}
\big|\widehat{\,\,I_{k,j}^L}(t,\xi)\big|\lesssim \int_{\R^2} e^{-\delta_0\langle\xi-\eta-\beta,\eta\rangle^{1/2}}\big|\widehat{\,\,F_{0k}^j}(\beta)\big|\, d\beta d\eta.
\end{equation}
The desired bounds \eqref{pmt10} for $I_k^L$ follow from \eqref{pmt32}. 

The enhanced dissipation estimates \eqref{intM3} then follow from \eqref{intM4}, Proposition \ref{LAP22}, bound \eqref{LAPM2.1}, and Theorem \ref{DSR0}.

Finally we notice that \eqref{intM5} follows from the formula for $v\in\R, t\ge0$,
\begin{equation}\label{pmt35}
\Phi_k(t,v)=\int_\R \mathcal{G}_k(v,v') F_k(t,v')e^{-ikt(v'-v)}\,dv'
\end{equation}
and the bounds on the Green's function $\mathcal{G}_k(v,v')$, see Lemma \ref{pre200}. The proof of Theorem \ref{intM1} is now complete. 

\medskip

\begin{center}{\bf Acknowledgement}\end{center}

We are grateful to Alexandru Ionescu for invaluable discussions during the project. We also thank Bernard Helffer for comments which improved the presentation of the paper.

\end{document}